\documentclass[a4paper,11pt]{amsart}

\usepackage{amssymb}
\usepackage{amsfonts}
\usepackage{amsmath}

\usepackage{graphicx}
\usepackage[all]{xy}
\usepackage{array}

\usepackage{amsthm}

\newtheorem{theorem}{Theorem}[section]

\newtheorem{lemma}[theorem]{Lemma}
\newtheorem{proposition}[theorem]{Proposition}
\newtheorem{corollary}[theorem]{Corollary}

\theoremstyle{definition}
\newtheorem{definition}[theorem]{Definition}

\newtheorem{observation}[theorem]{Observation}

\newtheorem{remark}[theorem]{Remark}

\newtheorem{example}[theorem]{Example}

\theoremstyle{remark}

\newcommand{\Si}{\mathfrak{S}}
\newcommand{\Ext}{{\mathrm{Ext}}}
\renewcommand{\hom}{{\mathrm{Hom}}}
\newcommand{\End}{{\mathrm{End}}}
\newcommand{\Id}{{\mathrm{Id}}}
\renewcommand{\k}{\Bbbk}

\newcommand{\op}{\mathrm{op}}

\renewcommand{\P}{\mathcal{P}}
\newcommand{\V}{\mathcal{V}}
\newcommand{\A}{\mathcal{A}}
\newcommand{\B}{\mathcal{B}}
\newcommand{\C}{\mathcal{C}}
\newcommand{\Z}{\mathbb{Z}}
\newcommand{\E}{\mathbb{E}}

\newcommand{\N}{\mathcal{N}}
\newcommand{\K}{\mathcal{K}}
\renewcommand{\L}{\mathcal{L}}
\newcommand{\F}{\mathcal{F}}
\newcommand{\diag}{\mathrm{diag}\,}

\newcommand{\ohat}{\widehat{\otimes}}
\newcommand{\otens}{\overline{\otimes}}
\newcommand{\Tot}{\mathrm{Tot}\,}

\newcommand{\U}{\mathcal{U}}
\newcommand{\Fct}{\mathrm{Fct}}
\renewcommand{\H}{\mathbb{H}}
\newcommand{\wtheta}{\widetilde{\Theta}}
\newcommand{\wlax}{\widetilde{\lax}}
\newcommand{\lax}{\square}

\newcommand{\MC}{\overline{\C}}
\newcommand{\oshuffle}{\overline{\nabla}}

\newcommand{\Ch}{\mathrm{Ch}}
\newcommand{\Inj}{\mathrm{Inj}}
\newcommand{\Proj}{\mathrm{Proj}}
\newcommand{\RR}{\mathbf{R}}
\newcommand{\DD}{\mathbf{D}}

\newcommand{\KK}{\mathbf{K}}

\input cyracc.def

\title{Ringel duality and derivatives of non-additive functors}
\author{Antoine Touz\'e }

\begin{document}

\sloppy

\maketitle
\begin{abstract}
We prove that Ringel duality in the category strict polynomial functors can be interpreted as derived functors of non-additive functors (in the sense of Dold and Puppe). We give applications of this fact for both theories.\\

\end{abstract}
\section{Introduction}

The purpose of this paper is to give an explicit and simple relation between two fundamental theories: Ringel duality for the representations of Schur algebras, and the theory of derived functors of non-additive functors.

Derived functors of non additive functors were introduced by Dold and Puppe in \cite{DP1,DP2}, and later generalized by Quillen \cite{Quillen}. These derived functors arose from the work on simplicial structures (the Dold-Kan correspondence), and came as a conceptual framework explaining many computations from algebraic topology. For example, the `quite bizarre' functors $\Omega(\Pi)$ and $R(\Pi)$ discovered by Eilenberg and Mac Lane in their study \cite{EML2} of the low degrees of the integral homology of Eilenberg-Mac Lane spaces can be interpreted as the derived functors of the second symmetric power $S^2:\mathrm{Ab}\to \mathrm{Ab}$. More generally, the whole homology of Eilenberg-Mac Lane spaces can be described as derived functors of the symmetric algebra functor $S^*$ or of the group ring functor. Derived functors are also related to the homology of symmetric products \cite{Do,DP2}, with group theory and the homotopy groups of Moore spaces (including the homotopy groups of spheres) \cite{Curtis,MP,BM}. 

Representations of Schur algebras arose in a quite different context, from the works of Schur on the general linear group \cite{Schur}. Modules over Schur algebras correspond to polynomial representations of $GL_n$. Representation theory of Schur algebras has known a great development in the eighties and the nineties, with the work of Akin, Buchsbaum and Weyman on characteristic free representation theory see e.g. \cite{ABW,AB1,AB}, the development of the theory highest weight categories, see e.g. \cite{CPS,Ringel}, the work of Donkin, Green and many others (we refer the reader to the books \cite{Green,Martin} for further references). A recent development of representations of Schur algebras is the introduction of the categories $\P_{d,\k}$ of strict polynomial functors by Friedlander and Suslin. Such functors can be thought of as functors $F:\V_\k\to \V_\k$ (where $\V_\k$ is the category of finitely generated projective $\k$-modules) defined by polynomial formulas (which are homogeneous of degree $d$).  Typical examples are the symmetric powers $S^d$, the exterior powers $\Lambda^d$ or the divided powers $\Gamma^d$. Friedlander and Suslin proved \cite[Thm 3.2]{FS} an equivalence of categories $$\P_{d,\k}\simeq S(n,d)\text{-mod}$$
between strict polynomial functors and modules over the Schur algebra $S(n,d)$, provided $n\ge d$. So, problems involving Schur algebras may be approached via strict polynomial functors. This point of view has been very fruitful for cohomological computations, see e.g. \cite{FS,FFSS,Chalupnik2,TouzeEML}, and it is the point of view which we adopt in the article. 

The category of strict polynomial functors is equipped with a duality operation $\Theta$ which produces for each 
functor $F$ a `signed version' $\Theta F$ of $F$. For example, $\Theta S^d=\Lambda^d$. Actually, this duality has better properties at the level of derived categories, where it becomes an equivalence of categories. This equivalence of categories was first studied in the framework of highest weight categories and representations of Schur algebras in \cite{Ringel,DonkinKos} and it is commonly called Ringel duality (although it has been called `Koszul duality'\footnote{A first version of this article was entitled `Koszul duality and derivatives of non-additive functors' because the author was not aware of the name `Ringel duality'. Since $\Theta$ is not Koszul duality in the sense of \cite{Priddy, BGS}, the title has been modified to avoid confusion.} in the framework of strict polynomial functors in \cite{Chalupnik2}).

\subsection*{Main result.} Let us now state our main result. Let $\k$ be a PID, and let $F\in\P_{d,\k}$ be a strict polynomial functor. If $V\in\V_\k$ is a free finitely generated $\k$-module, we build an isomorphism:
$$L_{nd-i}F(V;n)\simeq H^i\left(\Theta^n F(V)\right) \quad (*)$$
The objects on the left hand side  of the isomorphism are the derived functors of $F$ (in the sense of Dold and Puppe) and the objects on the right hand side are the homology groups of $n$-th iterated Ringel dual of $F$. 

Actually, our main theorem \ref{thm-main} is sharper: it asserts a version of isomorphism $(*)$ in the derived category, and it describes the slightly delicate compatibility of the isomorphism with tensor products.

If $n=1$ or $n=2$, the homology groups of $\Theta^n F$ can be interpreted as extension groups, so the isomorphism $(*)$ takes a more concrete form. For example, for $V=\k$, we have isomorphisms, natural in $F$:
$$L_{d-i}F(\k;1)\simeq \Ext^i_{\P_{d,\k}}(\Lambda^d,F)\,,\quad L_{2d-i}F(\k;2)\simeq \Ext^i_{\P_{d,\k}}(S^d,F) $$

The existence of the isomorphism $(*)$ was hinted at in \cite{TouzeEML}, where the author proved (without appealing to Ringel duality) that the extension groups $\Ext^*_{\P_{\k}}(\Lambda^*,\Gamma^*)$ and $\Ext^*_{\P_{\k}}(S^*,\Gamma^*)$ are related to the singular homology with $\k$ coefficients of the Eilenberg-Mac Lane spaces $K(\Z,3)$ and $K(\Z,4)$, which are respectively given by $L_*\Gamma^*(\k;1)$ and $L_*\Gamma^*(\k;2)$. 

Isomorphism $(*)$ yields other unexpected relations between representation theory of Schur algebras and algebraic topology. For example, let $\L^*(\Z^m)$ denote the free Lie algebra generated by $\Z^m$. Each degree of the free Lie algebra yields a strict polynomial functor $\L^d\in\P_{d,\Z}$, and their derived functors appear in the first page of the Curtis spectral sequence \cite{Curtis,MP,BM}, which converges to the homotopy of Moore spaces. For instance, there are spectral sequences converging to the unstable homotopy groups of the spheres $\mathbb{S}^2$ and $\mathbb{S}^3$:
\begin{align*}&E^1_{i,j}=\Ext^{j-i}_{\P_{j,\Z}}(\Lambda^j,\L^j)\Longrightarrow \pi_{j+1}(\mathbb{S}^2)\;,\\&  E^1_{i,j}=\Ext^{2j-i}_{\P_{j,\Z}}(S^j,\L^j)\Longrightarrow \pi_{j+1}(\mathbb{S}^3)\;.\end{align*}

Our result is also interesting from a computational point of view, because people working on each side of isomorphism $(*)$ use different techniques and understand different phenomena. Specialists of Schur algebras have developed combinatorics of partitions, highest weight categories, or use results coming from algebraic groups and algebraic geometry (e.g. Kempf theorem \cite[Part II, Chap 4]{Jantzen}). On the other side, homotopists use simplicial techniques, with an intuition coming from topology. Results which are well-known in the world of representations of Schur algebras translate into results which are unknown to homotopists and vice-versa.
So we hope that our results can serve as a basis for fruitful new interactions between these two subjects.

\subsection*{Organization of the paper.} The paper is more or less self-contained, and we have tried to give an elementary treatment of the subject. 

The first three sections are mainly introductory. Section \ref{sec-str} is an introduction to Schur algebras, strict polynomial functors and their derived categories. Section \ref{section-Kos} is a presentation of Ringel duality, based on Cha{\l}upnik's definition from \cite{Chalupnik2}. Finally, section \ref{section-DP} recalls the definition of derived functors of non-additive functors, and extends the classical definition to the derived category of strict polynomial functors.

In section \ref{section-main}, we prove our main theorem, which gives the link between Ringel duality and derived functors of non additive functors.

In section \ref{sec-applic}, we give some applications of our main theorem. All these applications follow the same principle: to obtain a theorem on one side of the isomorphism $(*)$, we prove a statement on the other side and we translate it using isomorphism $(*)$. In this way we obtain the following new results.
\begin{enumerate}
\item We prove a d\'ecalage formula for derived functors, generalizing the usual d\'ecalage formula due to Quillen \cite{Quillen2} and Bousfield \cite{Bousfield}, and which also generalizes computations of Bott \cite{Bott}.
\item  We find a formula to compute Ringel duals of plethysms (i.e. composites of functors). This formula yields many $\Ext$-computations in $\P_\k$. In particular it provides obstructions to the existence of certain `universal filtrations' (the non-existence of these filtrations was conjectured in \cite{Boffi}).
\item We show how to use block theory of Schur algebras to prove vanishing results for derived functors.
\item We supplement these applications by giving an example of a computation of derived functors using Ringel duality (thus retrieving a result of \cite{BM}).
\end{enumerate}

For the sake of completeness, we have included: a conversion table to help the reader read the literature on Schur functors in section \ref{subsubsec-schur}, a short discussion about the target category of our strict polynomial functors in section \ref{sec-arb}, and an appendix providing an elementary proof that the category of strict polynomial functors is equivalent to the category of modules over the Schur algebra.

\section{Strict polynomial functors and their derived categories}\label{sec-str}

This section is an introduction to the theory of strict polynomial functors originally developed in \cite[Sections 2 and 3]{FS} (and in \cite[Section 2]{SFB} over an arbitrary commutative ring). We first give the basics of the theory over an arbitrary commutative base ring $\k$. When $\k$ is a Dedekind ring (e.g. a PID), the category has an internal $\hom$ and we recall its properties. Finally, we give a brief account of the derived category of strict polynomial functors over a PID. 

\subsection{Recollections of Strict Polynomial Functors}\label{subsec-SPF}
Strict polynomial functors were originally defined over a field $\k$ in \cite{FS} but as remarked in \cite{SFB}, the generalization over a commutative ring $\k$ is straightforward. In this more general setting, we let $\V_\k$ be the category of finitely generated projective $\k$-modules and $\k$-linear maps. (This notation and the letters $V$, $W$, etc. denoting the objects of $\V_\k$ come from the field case where $\V_\k$ is the category of finite dimensional vector spaces).

\subsubsection{Schur algebras, categories $\Gamma^d\V_\k$, and their representations}\label{subsubsec-SGR}

Let $\k$ be a commutative ring and let $n$ and $d$ be positive integers. The tensor product ${\k^n}^{\,\otimes d}$ is acted on by the symmetric group $\Si_d$ which permutes the factors of the tensor product. The Schur algebra $S(n,d)$ is the algebra $\End_{\Si_d}({\k^n}^{\,\otimes d})$ of $\Si_d$-equivariant $\k$-linear homomorphisms \cite[2.6c]{Green}.

The categories $\Gamma^d\V_\k$ generalize Schur algebras. The objects of $\Gamma^d\V_\k$ are the finitely generated projective $\k$-modules. The homomorphism modules $\hom_{\Gamma^d\V_\k}(V,W)$ are the $\Si_d$-equivariant maps from $V^{\otimes d}$ to $W^{\otimes d}$ and the composition is just the composition of $\Si_d$-equivariant maps. In particular, $S(n,d)$ is nothing but the algebra of endomorphisms of the object $\k^n\in\Gamma^d\V_\k$. By abuse, we identify $S(n,d)$ with the full subcategory of $\Gamma^d\V_\k$ with $\k^n$ as unique object. 

The category of $\k$-linear representations of $\Gamma^d\V_\k$ in $\V_\k$ (i.e. the category of functors $F:\Gamma^d\V_\k\to\V_\k$ whose action on morphisms $f\mapsto F(f)$ is $\k$-linear) is denoted by $\P_{d,\k}$ and is commonly called the category of degree $d$ homogeneous strict polynomial functors\footnote{This is not the original definition of Friedlander and Suslin, but it is not hard to show it is equivalent, see e.g. \cite[p. 41]{FPan}.}. Restriction of $\Gamma^d\V_\k$-representations to the full subcategory with unique object $\k^n$ yields a functor 
$$\P_{d,\k}\to S(n,d)\text{-mod}\;.$$
(Here $S(n,d)\text{-mod}$ stands for the category of modules over the Schur algebra, which are finitely generated and projective as $\k$-modules.) Friedlander and Suslin proved \cite{FS,SFB} that it is an equivalence of categories if $n\ge d$ (we give a direct proof of this in appendix \ref{app}). 

\begin{example}
The strict polynomial functor $\otimes^d$ sends an object $V$ of $\Gamma^d\V_\k$ to $V^{\otimes d}$ and sends a morphism $f\in \hom_{\Gamma^d\V_\k}(V,W)$ to the same $f$, but viewed as an element of  $\hom_\k(V^{\otimes d},W^{\otimes d})$.
\end{example}

Sums, products, kernels, cokernels, etc. in the category $\P_{d,\k}$ are computed objectwise in the target category $\V_\k$, so that the structure of $\P_{d,\k}$ inherits many properties from $\V_\k$. In particular, if $\k$ is a field, then $\P_{d,\k}$ is an abelian category, and more generally over an arbitrary commutative ring $\k$, $\P_{d,\k}$ is an exact category in the sense of Quillen \cite{Buehler,Keller}, the admissible short exact sequences $F\xrightarrow[]{\iota} G\xrightarrow[]{p} H $ (i.e. the conflations $(\iota,p)$ according to the terminology of \cite{Keller}) being the ones which become short exact sequences after evaluation on any $V\in\V_\k$.

We denote by $\P_{0,\k}$ the category of constant functors from $\V_\k$ to $\V_\k$. The category $\P_{\k}$ of strict polynomial functors of finite degree is defined by
$$\P_\k=\textstyle\bigoplus_{d\ge 0}\P_{d,\k}\;,$$
where the right hand side term denotes the subcategory of $\Pi_{d\ge 0} \P_{d,\k}$ whose objects are the finite products (only a finite number of terms are non zero). If $F=\oplus_{i\le d} F_i$, with $F_i\in\P_{i,\k}$ and $F_d\ne 0$, then $F$ is said to have strict polynomial degree $d$. 
 
\subsubsection{Strict polynomial functors vs ordinary functors}

There is an exact forgetful functor from the category $\P_{\k}$ to the category $\F_\k$ of ordinary functors from $\V_\k$ to $\k$-modules:
$$\U:\P_\k\to \F_\k=\Fct(\V_\k,\k\text{-mod})\;,$$
defined in the following way. If $F\in\P_{0,\k}$ then $F$ is already a (constant) functor. If $d\ge 1$, the functor $\U$ sends an element $F\in\P_{d,\k}$ to the precomposition $F\circ \gamma^d$, where $\gamma^d:\V_\k\to \Gamma^d\V_\k$ is the functor which is the identity on objects and sends a map $f$ to $f^{\otimes d}$.

\begin{example}
$\U(\otimes^d)$ is the functor $V\mapsto V^{\otimes d}$ in the usual sense. 
\end{example}

\begin{remark}
One can prove that all the functors in the image of $\U$ are polynomial in the sense of Eilenberg and Mac-Lane \cite{EML2}, and the Eilenberg Mac-Lane degree of $\U F$ is less or equal to the strict polynomial degree of $F$. The inequality may very well be strict, e.g. if $\k=\mathbb{F}_p$, the Frobenius twist functor $I^{(1)}$ \cite[(v) p. 224]{FS} has strict polynomial degree $p$, but $\U I^{(1)}$ is the identity functor, whose Eilenberg Mac Lane degree equals one. 
\end{remark}

\subsubsection{Further structure of $\P_\k$ over a commutative ring $\k$}

We now briefly summarize further structure borne by the category of strict polynomial functors over a commutative ring $\k$.
\begin{description}
\item[Tensor products.] One can take tensor products of strict polynomial functors in the target category $(F\otimes G)(V)
:=F(V)\otimes G(V)$. This yields bi-exact functors:
$$\otimes:\P_{d,\k}\times\P_{e,\k}\to \P_{d+e,\k}\quad\text{ and }\quad\otimes:\P_\k\times\P_\k\to \P_\k\;. $$
\item[Composition.] One can also compose strict polynomial functors: $(F\circ G)(V)=F(G(V))$. In this way we obtain exact functors:
$$\circ:\P_{d,\k}\times\P_{e,\k}\to \P_{de,\k}\quad\text{ and }\quad\circ:\P_\k\times\P_\k\to \P_\k\;. $$
\item[Duality.] We let $V^\vee$ be the $\k$-module $\hom_\k(V,\k)$. The formula $F^\sharp(V):=F(V^\vee)^\vee$ defines equivalences of categories (which are self inverse):
$$^\sharp: (\P_{d,\k})^\op\xrightarrow[]{\simeq}\P_{d,\k}\quad\text{ and }\quad \,^\sharp:(\P_\k)^\op\xrightarrow[]{\simeq}\P_\k\;.$$
\item[Projectives.]
We denote by $\Gamma^{d,V}$ the strict polynomial functor defined by (the first equality is the definition, and the last two equalities are the canonical identifications):
$$\Gamma^{d,V}(W)=\hom_{\Gamma^d\V_\k}(V,W)=(\hom_\k(V,W)^{\otimes d})^{\Si_d}=\Gamma^{d}(\hom_\k(V,W))\;.$$
The Yoneda lemma yields an isomorphism (natural in $V$ and $F$):
$$\hom_{\P_{d,\k}}(\Gamma^{d,V},F)\simeq F(V)\;,$$ 
so the $\Gamma^{d,V}$ are actually projectives of $\P_{d,\k}$.
In fact, the functors $\Gamma^{d,V}$ for $V\in\V_\k$, form a projective \emph{generator} of $\P_{d,\k}$. To be more specific, the canonical map $\Gamma^{d,V}(W)\otimes G(V)\to G(W)$ is an epi if $V$ contains $\k^d$ as a direct summand, see \cite[Thm 2.10]{FS} or appendix \ref{app}.

We simply denote by $\Gamma^d$ the functor $\Gamma^{d,\k}$ (if $d=0$, it is the constant functor with value $\k$). For all $n$-tuples $\mu=(\mu_1,\dots,\mu_n)$, we denote by $\Gamma^\mu$ the tensor product $\bigotimes_{i=1}^n \Gamma^{\mu_i}$. The functor $\Gamma^{d,\k^n}$ decomposes as the direct sum $\bigoplus_\mu \Gamma^\mu$, the sum being taken over all $n$-tuples $\mu$ of nonnegative integers of weight $\sum\mu_i=d$. This has two consequences. First, the $\Gamma^\mu$ form a projective generator. Second, the tensor product of projectives is once again projective.

\item[Injectives.] By duality, the functors $S^d_V = (\Gamma^{d,V})^\sharp:W\mapsto S^d(W\otimes V)$ form an injective cogenerator of $\P_{d,\k}$. We denote by $S^d$ the functor $S^{d}_\k$ and for all $n$-tuples $\mu$ of nonnegative integers we denote by $S^\mu$ the tensor product $\otimes_{i=1}^n S^{\mu_i}$. All these functors are injectives and the family $(S^\mu)_\mu$ indexed by tuples of weight $d$ forms an injective cogenerator of $\P_{d,\k}$.

\item[Extensions.] Since the $\P_{d,\k}$ (hence $\P_\k$) are exact categories with enough projectives and injectives, there is no problem in defining extension groups \cite{Buehler,Keller}. Since $\P_\k=\textstyle\bigoplus_{d\ge 0}\P_{d,\k}$, if $F$ and $G$ are homogeneous strict polynomial functors then $\Ext^*_{\P_\k}(F,G)$ equals $\Ext^*_{\P_{d,\k}}(F,G)$ if $F$ and $G$ have the same degree, and zero otherwise.
\end{description}

\begin{remark}\label{rk-PT}
The category $\P_{d,\k}$ is a full exact subcategory of the abelian category $\widetilde{\P}_{d,\k}$ of functors from $\Gamma^d\V_\k$ to arbitrary $\k$-modules. 
A projective resolution $\P_{d,\k}$ yields a projective resolution in $\widetilde{\P}_{d,\k}$, so the inclusion functor $\P_{d,\k}\hookrightarrow \widetilde{\P}_{d,\k}$ induces an isomorphism (see also section \ref{sec-arb}):
$$\Ext^*_{\P_{\k,d}}(F,G)\simeq \Ext^*_{\widetilde{\P}_{\k,d}}(F,G)\;. $$
In this article, we prefer to work in $\P_{d,\k}$ because one loses duality and exactness of tensor products for functors with arbitrary values.
\end{remark}

\subsection{The internal $\hom$ over Dedekind rings}

In this section, we restrict our attention to the category $\P_{d,\k}$ of homogeneous strict polynomial functors of degree $d$. We assume that $\k$ is a Dedekind ring (for example a PID). It ensures that $\hom$-groups in $\P_{d,\k}$ are finitely generated and projective $\k$-modules. Thus, we can introduce a parameter to obtain internal $\hom$s. This idea is already used e.g. in \cite{Chalupnik2,TouzeEML}. We recall here their definition and main properties. We begin with a few notations about parameterized functors. 

\subsubsection{Parameterized functors}\label{subsubsec-param}
If $F$ is a strict polynomial functor and $V\in \V_\k$, we form a strict polynomial functor $F^V$ with parameter $V$, by letting $F^V(W)=F(\hom_\k(V,W))$. It is straightforward to check that this actually yields a functor:
$$(\Gamma^d\V_\k)^{\op}\times\Gamma^d\V_\k\to \V_\k,\quad (V,W)\mapsto F^V(W).$$
The parameter `$V$' is written as an exponent to indicate the contravariance in $V$.
The notation for parameterized functors is coherent with the notation for projectives: $(\Gamma^d)^V=\Gamma^{d,V}$, and more generally, $(\Gamma^{d,V_1})^{V_2}=\Gamma^{d,V_1\otimes V_2}$. 

Similarly, one can introduce a covariant parameter $V$ (hence written as an index), by letting $F_V(W)=F(V\otimes W)$. Observe that $(F_V)^{\sharp}\simeq (F^\sharp)^V$, and that this notation agrees with the notation for injectives: $(S^d_{V_1})_{V_2}\simeq S^d_{V_1\otimes V_2}$.
We shall use heavily the following fact in the sequel.
\begin{observation}
We can replace the parameter $V$ by a simplicial free $\k$-module $X$. In this case $F_X$ becomes a  simplicial object in $\P_{d,\k}$, with $(F_X)_n=F_{X_n}$ with face operators  $F(d_i):F_{X_n}\to F_{X_{n-1}}$ and degeneracy operators $F(s_i):F_{X_n}\to F_{X_{n+1}}$. Similarly $F^X$ becomes a cosimplicial object in $\P_{d,\k}$. This generalizes to multisimplicial objects: if $X$ is a bisimplicial free $\k$-module then $F_X$ is a bisimplicial strict polynomial functor, and so on.
\end{observation}

\subsubsection{The internal $\hom$ in $\P_{d,\k}$}\label{subsubsec-internal}
We recall the construction and properties of the internal $\hom$ in $\P_{\k,d}$. For the proofs (which are elementary) we refer the reader to \cite[section 4]{TouzeEML}.
\begin{definition} Let $\k$ be a Dedekind ring. Then for all $F,G\in\P_{d,\k}$, the $\k$-module $\hom_{\P_\k}(F,G)$ is finitely generated and projective. In particular, the formula 
$V\mapsto \hom_{\P_\k}(F^V,G)$ defines an element of $\P_{d,\k}$, which we denote by $\H(F,G)$, and we have a bifunctor:
$$\begin{array}{cccc} 
\H:&\P_{d,\k}^\op\times\P_{d,\k}&\to&\P_{d,\k}\\
& (F,G)&\mapsto & \H(F,G) 
\end{array}\;.
$$
\end{definition}

The bifunctor $\H$ enjoys the following properties. First, the Yoneda lemma and duality respectively yield isomorphisms (natural in $F,G,U$):
$$\H(\Gamma^{d,U},G)\simeq G_U\;,\qquad \H(F,G)\simeq \H(G^\sharp,F^\sharp)\;.$$
The bifunctor $\H$ is also compatible with tensor products. For $i=1,2$, let $F_i,G_i$ be homogeneous strict polynomial functors of degree $d_i$. Tensor products induce a morphism of strict polynomial functors (natural in $F_i, G_i$): 
$$\H(F_1,G_1)\otimes\H(F_2,G_2)\xrightarrow[]{\otimes} \H(F_1\otimes F_2,G_1\otimes G_2)\;.$$
Finally, if $X^*$ denotes $S^*,\Lambda^*$ or $\Gamma^*$, we may postcompose this tensor product by the map induced by the comultiplication $X^{d_1+d_2}\to X^{d_1}\otimes X^{d_2}$ to get an isomorphism \cite[Lemma 5.5]{TouzeEML}:
\begin{align*}\H(X^{d_1},G_1)\otimes\H(X^{d_2},G_2)\xrightarrow[]{\simeq} \H(X^{d_1+d_2},G_1\otimes G_2)\;.\end{align*}
\begin{remark}
The bifunctor $\H$ is an internal $\hom$ in the usual sense. In particular, it is adjoint to the symmetric monoidal product $\bullet$ in $\P_{\k,d}$ defined by $F\bullet G=\H(F,G^\sharp)^\sharp$. However, the facts recalled above will be sufficient for our purposes, and we refer the reader to \cite{Krause} for a nice exposition of this symmetric monoidal product. We have used the notation `$\H$' instead of the more standard notation $\underline{\hom}_{\P_{d,\k}}$ for typographical reasons (to keep formulas compact). 
\end{remark}

\subsection{Derived categories}

Now we assume that $\k$ is a PID. This ensures \cite{DonkinHDim,AB} that the category of modules over the Schur algebra, hence the category $\P_{d,\k}$, has finite homological dimension. So we can work without trouble in the bounded derived category. The assumption on $\k$ also allows us to use internal $\hom$s.

\subsubsection{Chain complexes, etc.}\label{subsec-conventions}

We recall the conventions and notations for complexes which we use in the article. In what follows, $\A$ is an additive category, enriched over a commutative ring $\k$ (e.g. the category of $\k$-modules, $\P_{d,\k}$, modules over a $\k$-algebra, etc.)

\begin{description}
\item[Complexes.] We let $\Ch(\A)$ be the category of complexes in $\A$. Gradings are indifferently denoted using the homological convention or the cohomological one. As usual, the conversion between the two conventions is realized by the formula $C^i=C_{-i}$. We denote by $\Ch^b(\A)$, $\Ch^+(\A)$, $\Ch^-(\A)$ the full subcategories of bounded, bounded below, and  bounded above cochain complexes. We also denote by $\Ch_{\ge 0}$ the full subcategory of $\Ch^-(\A)$ consisting of nonnegatively graded chain complexes.
We let $\KK(\A)$, $\KK^b(\A)$, $\KK^+(\A)$, $\KK^-(\A)$, $\KK_{\ge 0}$ the corresponding homotopy categories.

\item[Suspension.] If $C$ is a complex, its $i$-th suspension $C[i]$ is defined by $C[i]^n=C^{n-i}$ (or $C[i]_n=C_{i+n}$) and $d_{C[i]}= (-1)^{i}d_C$. The suspension of a morphism of complexes $f:C\to D$ is given by $f[d]_i=f_{d+i}$

\item[Tensor products.] If $\A$ is equipped with a symmetric, associative and unital tensor product (in other words: $\A$ is a strict symmetric monoidal category), then so are the categories $\Ch^b(\A)$, $\Ch^+(\A)$, $\Ch^-(\A)$ and $\Ch_{\ge 0}(\A)$. Differentials in tensor products of complexes are defined as usual by the Koszul convention:
$$ d(x\otimes y)=dx\otimes y + (-1)^{\deg x}x\otimes dy\;,$$
and we denote by $\tau$ the symmetry isomorphism $C\otimes D\simeq D\otimes C$, which maps $x\otimes y$ to $(-1)^{\deg x\deg y}y\otimes x$.

\item[Bicomplexes.] A bicomplex $B$ is a bigraded object in $\A$, equipped with two differentials $d_1:B^{i,j}\to B^{i+1,j}$ and $d_2:B^{i,j}\to B^{i,j+1}$ which commute. The total complex $\Tot B$ associated to $B$ is defined as usual by using the Koszul sign convention. Thus $(\Tot B)^k=\bigoplus_{i+j=k}B^{i,j}$ and for $x\in B^{i,j}$, $dx=d_1x+(-1)^id_2x$.
\end{description}
\subsubsection{Quasi-isomorphisms and derived categories}\label{subsec-derived}
Now we assume that $\k$ is a PID. A chain map $f:C\to D$ is a quasi-isomorphism if it satisfies one of the conditions of the following lemma.
\begin{lemma}\label{lm-caracqis}
Let $\k$ be a PID. 
Let $C,D\in\Ch(\P_{d,\k})$ and let $f:C\to D$ be a chain map. The following conditions are equivalent. 
\begin{itemize}
\item[(i)] For all $V\in\V_\k$, the morphism of complexes of $\k$-modules $f_V:C(V)\to D(V)$ induces an isomorphism in homology. 
\item[(ii)] For all $V\in\V_\k$, the mapping cone $M(V)$ of $f_V$ is a complex of $\k$-modules with trivial homology.
\item[(iii)] The  mapping cone $M$ of $f$ is a complex of strict polynomial functor, which decomposes as the Yoneda splice of admissible short exact sequences $Z^n\hookrightarrow M^n \twoheadrightarrow Z^{n+1}$.
\end{itemize}
\end{lemma}
\begin{proof}
The equivalence between (i) and (ii) is standard \cite[Cor 1.5.4]{Weibel}. It is trivial that (iii) implies (ii). The converse uses that $\k$ is a PID. Since $\k$ is a PID, the cycles of a complex are functors from $\Gamma^d\V_\k$ to $\V_\k$, i.e. they are genuine strict polynomial functors. The  
$Z^n\hookrightarrow M^n \twoheadrightarrow Z^{n+1}$ are admissible short exact sequences since they are exact after evaluation on $V\in\V_\k$. 
\end{proof}

\begin{remark}
Condition (i) mimics the definition quasi-isomorphism in abelian categories, and it is handy since it allows traditional spectral sequence argument to check that a map is a quasi-isomorphism. Condition (iii) is the standard definition (see \cite[Def. 10.16]{Buehler} or \cite[Section 11]{Keller}) of quasi-isomorphisms in exact categories (like $\P_{d,\k}$). 
\end{remark}

Let $*$ denote the symbol $+$, $-$ or $b$ or the empty symbol. Since $\P_{d,\k}$, resp. $\P_\k$ are exact categories, there are associated derived categories 
$\DD^*(\P_{d,\k})$, resp. $\DD^*(\P_{\k})$ \cite[Section 11]{Keller}, or \cite[Section 10.4]{Buehler}. The derived categories are localizations of $\KK^*(\P_{d,\k})$, resp. $\KK^*(\P_{\k})$, with respect to quasi-isomorphisms, exactly as in the case of an abelian category \cite[Chap 10]{Weibel}.

So, the objects of the derived categories
$\DD^*(\P_{d,\k})$, resp. $\DD^*(\P_{\k})$ are the same as the ones of $\Ch^*(\P_{d,\k})$, resp. $\Ch^*(\P_{\k})$. Morphisms in the derived categories, with source $C$ and target $D$ are represented by diagrams $C\rightarrow D' \leftarrow D$, where the first map is a chain map and the second map is a quasi-isomorphism. Two diagrams $C\to D'\leftarrow D$ and $C\to D''\leftarrow D$ represent the same morphism if and only if they fit into a commutative diagram (where the vertical arrows are quasi-isomorphisms):
$$\xymatrix@R=0.3cm{
& D'\ar[d] &\\
C\ar[ru]\ar[r]\ar[rd] &D'''& D\ar[lu]\ar[l]\ar[ld]\\
& D''\ar[u] &
}.$$

As in the case of abelian categories, the derived categories $\DD^*(\P_{d,\k})$ and $\DD^*(\P_{\k})$ are triangulated categories. The exact triangles are (rotates of) the ones isomorphic to the standard triangle $C\to D\to M\to C[-1] $ where $M$ denotes the mapping cone of the morphism $C\to D$.

Let the symbol $\ast$ stand for $+$, $-$ or $b$. By bi-exactness, tensor products of complexes induce tensor products at the level of the derived categories
\begin{align*}&\DD^*(\P_{d,\k})\times \DD^*(\P_{e,\k})\xrightarrow[]{\otimes} \DD^*({\P_{d+e,\k}})\,,\\
&\DD^*(\P_{\k})\times \DD^*(\P_{\k})\xrightarrow[]{\otimes} \DD^*({\P_{\k}})\;.
\end{align*}
These tensor products are symmetric, associative and unital. In particular, $\DD^*(\P_{\k})$ is a strict symmetric monoidal category.

Actually, we shall work mainly in bounded  derived categories. We list below some extra properties of these categories. 
\begin{description}
\item[Decomposition.] The derived category $\DD^b(\P_{\k})$ decomposes as the direct sum of its full subtriangulated categories $\DD^b(\P_{d,\k})$:
$$\DD^b(\P_\k)=\bigoplus_{d\ge 0}\DD^b(\P_{d,\k})\;.$$
\item[Duality.] By exactness, duality induce functors:
\begin{align*}&\DD^b(\P_{d,\k})\xrightarrow[]{^\sharp}\DD^b(\P_{d,\k}^\op)\simeq (\DD^b(\P_{d,\k}))^\op\;,\\
 &\DD^b(\P_{\k})\xrightarrow[]{^\sharp}\DD^b(\P_{\k}^\op)\simeq (\DD^b(\P_{\k}))^\op\;.
\end{align*}
\item[Injectives and projectives.] Since $\k$ is a PID, $\P_{d,\k}$ has finite homological dimension \cite{DonkinHDim,AB}. So each complex in $\DD^b(\P_{d,\k})$ or $\DD^b(\P_{\k})$ is quasi-isomorphic to a finite complex of injectives and to a bounded complex of projectives. Thus $\DD^b(\P_{d,\k})$ and $\DD^b(\P_{\k})$ are equivalent to their full subcategory of complexes of projective functors, and also to their full subcategory of injective objects. In particular we have equivalences of categories:
$$\DD^b(\P_\k)\simeq \KK^b(\Inj(\P_{\k}))\;,\quad \DD^b(\P_\k)\simeq \KK^b(\Proj(\P_{\k}))\;.$$
Since tensor products of injectives (resp. projectives) remain injective (resp. projective), the equivalences above may be realized by monoidal functors.
\item[Internal Hom.] Let $C\in\P_{d,\k}$. The internal $\hom$ functor induces a derived functor:
$$\RR\H(C,-): \DD^b(\P_{d,\k})\to \DD^b(\P_{d,\k})\;.$$
Tensor products induce morphisms natural with respect to the complexes $C,C',D,D'$, associative, commutative and unital:
$$\RR\H(C,D)\otimes \RR\H(C',D')\to \RR\H(C\otimes C',D\otimes D')\;. $$
If $D$ is a complex of $\H(C,-)$-acyclic objects, there is an isomorphism $\RR\H(C,D)\simeq \H(C,D)$. 

\end{description}

\section{Ringel duality}\label{section-Kos}
In this short section, we present Ringel duality. We adopt the point of view of Cha{\l}upnik \cite{Chalupnik2}, which we generalize over a PID $\k$. For an explanation of the combinatorial ideas encoded in Ringel duality, we refer the reader to \cite[Section 2]{Chalupnik2}.
\subsection{Definition of Ringel duality}

\begin{definition}[{\cite{Chalupnik2}}] Let $\k$ be a PID. The Ringel duality functor $\Theta$ is the triangle functor: 
$$\begin{array}{cccc}
\Theta: &\DD^b(\P_{d,\k})&\to &\DD^b(\P_{d,\k})\\
& C &\mapsto & \RR\H(\Lambda^d,C)
\end{array}.$$
\end{definition}

\begin{remark}
Functors of degree zero are constant functors, and $\Lambda^0$ is the constant functor with value $\k$. In particular, $\H(\Lambda^0,-)$ is the identity functor of $\P_{0,\k}$. So $\Theta$ is the identity map if $d=0$.
\end{remark}

Now we describe the compatibility of $\Theta$ with tensor products. For $C\in\DD^b(\P_{d,\k})$ and $D\in\DD^b(\P_{e,\k})$, the following composite, where the last map is induced by the comultiplication $\Lambda^{d+e}\to\Lambda^d\otimes\Lambda^e$, is an isomorphism:
$$\RR\H(\Lambda^d,C)\otimes\RR\H(\Lambda^e,D)\xrightarrow[]{\otimes}\RR\H(\Lambda^d\otimes\Lambda^e,C\otimes D)\to \RR\H(\Lambda^{d+e},C\otimes D)\;,$$
Indeed, it is true for complexes of injectives since in this case it reduces to the isomorphism $\H(\Lambda^d,C)\otimes\H(\Lambda^e,D)\xrightarrow[]{\simeq}\H(\Lambda^{d+e},C\otimes D)$ from section \ref{subsubsec-internal}. We denote this composite by $\lax$.
$$\lax:(\Theta C)\otimes (\Theta D)\xrightarrow[]{\simeq} \Theta(C\otimes D)\;.$$
The morphism $\lax$ is associative and unital since $\H(\Lambda^d,-)$ is. So, gathering all possible degrees $d$, we obtain the following result.
\begin{proposition}\label{prop-square}
Ringel duality yields a monoidal functor:
$$(\Theta,\square,\Id)\;:\; (\DD^b(\P_{\k}),\otimes,\k)\to (\DD^b(\P_{\k}),\otimes,\k) $$ 
\end{proposition}

\begin{remark}\label{rk-sgn-Ringel} Ringel duality is not a \emph{symmetric} monoidal functor. Indeed, the comultiplication $\Lambda^{d+e}\to \Lambda^d\otimes\Lambda^e$ composed with the permutation $\Lambda^d\otimes\Lambda^e\simeq \Lambda^e\otimes\Lambda^d$ equals $(-1)^{de}$ times the comultiplication. So the following diagram commutes up to a $(-1)^{de}$ sign.
$$\xymatrix{
\Theta(C)\otimes\Theta(D)\ar[r]_-\simeq^-{\lax}\ar[d]_-\simeq^-{\tau}& \Theta(C\otimes D)\ar[d]_-\simeq^-{\Theta(\tau)}\\
\Theta(D)\otimes\Theta(C)\ar[r]_-\simeq^-{\lax}& \Theta(D\otimes C)
}.$$
\end{remark}

\subsection{How to compute Ringel duals}
Let $C\in\DD^b(\P_{d,\k})$. Then $\Theta(C)$ is quasi-isomorphic to the complex $\H(\Lambda^d,D)$, where $D$ is a complex of $\H(\Lambda^d,-)$-acyclic objects quasi-isomorphic to $C$. The following lemma shows that many concrete functors are $\H(\Lambda^d,-)$-acyclic, so there is often a huge choice of complexes $D$ available for explicit computations.
\begin{lemma}\label{lm-acyclic}
The class of $\H(\Lambda^d,-)$-acyclic objects is stable under tensor products. It contains symmetric powers, exterior powers and more generally Schur functors. 
\end{lemma}
\begin{proof}
Stability of the class of $\H(\Lambda^d,-)$-acyclic objects by tensor product follows from the fact that $\lax$ is an isomorphism. Symmetric powers are $\H(\Lambda^d,-)$-acyclic because they are injective. For exterior powers, this well known fact is proved e.g. in \cite[Remark 7.5]{TouzeEML}. Finally, it is proved in \cite{AB} that Schur functors admit finite resolutions by tensor products of exterior powers. Hence they are $\H(\Lambda^d,-)$-acyclic\footnote{Alternative proofs of the $\H(\Lambda^d,-)$-acyclicity of Schur functors can be provided by the highest weight category structure of $\P_{d,\k}$ \cite[Thm 3.11]{CPS}, or by Kempf vanishing \cite[II, Prop 4.13]{Jantzen} combined with \cite[Cor 3.13]{FS}. For these alternative proofs, one must identify Schur functors in a slightly different context, see section \ref{subsubsec-schur}.}. 
\end{proof}
If $X^*$ denotes the symmetric exterior or divided power algebra and if $\lambda=(\lambda_1,\dots,\lambda_n)$ is a tuple of nonnegative integers of weight $\sum\lambda_i=d$, we denote by $X^\lambda$ the tensor product $\bigotimes_{i=1}^n X^{\lambda_i}$.
For the reader's convenience, we describe the (well-known) action of the functor $\H(\Lambda^d,-)$ on the $\H(\Lambda^d,-)$-acyclic objects $S^\lambda$ and $\Lambda^\lambda$ in the following lemma.
\begin{lemma}\label{lm-comput}
Let $d$ be a positive integer. The following computations hold in $\P_{d,\k}$.
\begin{enumerate}
\item[(i)] $\H(\Lambda^d,\otimes^d)\simeq \otimes^d$. 
If we denote by $\sigma:\otimes^d\to \otimes^d$ the morphism induced by a permutation $\sigma\in\Si_d$, then $\H(\Lambda^d,\sigma)= \epsilon(\sigma)\sigma$.
\item[(ii)] If $\lambda$ is a partition of weight $d$, $\H(\Lambda^d,S^\lambda)\simeq \Lambda^\lambda$. Moreover, the multiplication $\otimes^d\twoheadrightarrow S^\lambda$ is sent by $\H(\Lambda^d,-)$ to the multiplication $\otimes^d\twoheadrightarrow \Lambda^\lambda$.
\item[(iii)] If $\lambda$ is a partition of weight $d$, $\H(\Lambda^d,\Lambda^\lambda)\simeq \Gamma^\lambda$. Moreover, the comultiplication $\Lambda^\lambda\hookrightarrow\otimes^d$ is sent by $\H(\Lambda^d,-)$ to the comultiplication $\ \Gamma^\lambda\hookrightarrow\otimes^d$.
\end{enumerate}
\end{lemma}
\begin{proof}
Let us prove (i). For $d=1$ it reduces to the Yoneda isomorphism $\H(\Lambda^1,S^1)\simeq \Lambda^1$. For $d\ge 2$, we use the fact that $\lax$ is an isomorphism. 

(ii) is a little trickier. Using the isomorphism $\lax$, one reduces the proof of (ii) to the case of $S^\lambda=S^d$. By duality $\H(\Lambda^d,S^d)$ is isomorphic to $\H(\Gamma^d,\Lambda^d)$ which is isomorphic to $\Lambda^d$ by the Yoneda isomorphism. 

It remains to compute the image of the multiplication $\otimes^d\twoheadrightarrow S^d$ by $\H(\Lambda^d,-)$. By duality, this amounts to computing the map $\H(\otimes^d,\Lambda^d)\to \H(\Gamma^d,\Lambda^d)$ induced by the comultiplication $\Gamma^d\hookrightarrow\otimes^d$. First, $\Gamma^{d,\k^d}\simeq \bigoplus \Gamma^\lambda$, the sum being taken over all $d$-tuples $\lambda=(\lambda_1,\dots,\lambda_d)$ of nonnegative integers of weight $d$. Thus $\otimes^d$ identifies with a direct summand of $\Gamma^{d,\k^d}$, and one checks that the composite $\Gamma^d\hookrightarrow \otimes^d\hookrightarrow \Gamma^{d,\k^d}$ equals the map $\Gamma^d=\Gamma^{d,\k}\hookrightarrow \Gamma^{d,\k^d}$ induced by the map $\Sigma:\k^d\to \k$, $(x_1,\dots,x_d)\mapsto \sum x_i$. The map $\H(\Sigma,\Lambda^d)$ identifies through the Yoneda isomorphism with $\Lambda^d_{\Sigma}$. Now $\Lambda^*_\Sigma:\Lambda^*_{\k^d}\simeq \Lambda^{*\otimes d}\to \Lambda^*$ is the $d$-fold multiplication, so restricting our attention to the summand $\otimes^d$ of $\Lambda^d_{\k^d}$ we get the result.

Finally, to prove (iii), we first use the isomorphism $\lax$ to reduce the proof to the case of $\Lambda^\lambda=\Lambda^d$. Then  $\Lambda^d$ fits into an exact sequence $$\Lambda^d\hookrightarrow\textstyle\otimes^d\to \bigoplus_{i=1}^{d-1}(\otimes^{i-1})\otimes S^2\otimes (\otimes^{d-1-i})\;,$$
where the first map is the comultiplication and the components of the second map are obtained by tensoring the multiplication $\otimes^2\to S^2$ by identities. By (i), (ii) and by left exactness of $\H(\Lambda^d,-)$, $\H(\Lambda^d,\Lambda^d)$ is the kernel of the same map but with $S^2$ replaced by $\Lambda^2$. Hence it equals $\Gamma^d$, and the map $\H(\Lambda^d,\Lambda^d)\hookrightarrow \H(\Lambda^d,\otimes^d)$ identifies with the comultiplication $\Gamma^d\to \otimes^d$.
\end{proof}

As a consequence of the elementary computations of lemma \ref{lm-comput} we have:
\begin{lemma}\label{lm-iso}
The functor $\H(\Lambda^d,-)$ induces isomorphisms:
$$\hom_{\P_{d,\k}}(S^\mu,S^\lambda)\simeq \hom_{\P_{d,\k}}(\Lambda^\mu,\Lambda^\lambda)\simeq \hom_{\P_{d,\k}}(\Gamma^\mu,\Gamma^\lambda)\;. $$
These isomorphisms fit into a diagram:
$$\xymatrix{
\hom_{\P_{d,\k}}(S^\mu,S^\lambda)\ar[d]^{\sharp}\ar[r]^-{\simeq}&\hom_{\P_{d,\k}}(\Lambda^\mu,\Lambda^\lambda)\ar[d]^{\sharp}\\
\hom_{\P_{d,\k}}(\Gamma^\lambda,\Gamma^\mu)&\hom_{\P_{d,\k}}(\Lambda^\lambda,\Lambda^\mu)\ar[l]_{\simeq}
}.$$
\end{lemma}
\begin{proof}
We first treat the particular case $S^\lambda=S^\mu=\otimes^d$. In that case, the morphisms $\sigma:\otimes^d\to\otimes^d$ induced by the permutations $\sigma\in\Si_d$ (acting on $\otimes^d$ by permuting the factors of $\otimes^d$) form a basis of $\hom_{\P_{d,\k}}(\otimes^d,\otimes^d)$. So lemma \ref{lm-comput}(i) completely describes $\H(\Lambda^d,-)$ on tensor products, and the result follows.
 
Now we prove the general case. First, we check the commutativity of the diagram of lemma \ref{lm-iso}. Let $f\in\hom_{\P_{d,\k}}(S^\lambda,S^\mu)$. By projectivity of $\otimes^d$, the morphism $f$ fits into a commutative square (where the vertical arrows are induced by the multiplication)
$$\xymatrix{
\otimes^d\ar@{->>}[d]\ar[r]^-{\overline{f}}&\otimes^d\ar@{->>}[d]\\
S^\lambda\ar[r]^-{f} &S^\mu
}.$$
If we apply to this square on the one hand duality, and on the other hand first $\H(\Lambda^d,-)$, then duality and once again  $\H(\Lambda^d,-)$, we obtain commutative squares (where the vertical arrows are the comultipications)
$$\xymatrix{
\otimes^d&\otimes^d\ar[l]_-{\overline{f}^\sharp}\\
\Gamma^\lambda\ar@{^{(}->}[u] &\Gamma^\mu\ar@{^{(}->}[u]\ar[l]_-{f^\sharp}
},
\quad\;
\xymatrix{
\otimes^d&&&\otimes^d\ar[lll]_-{\H(\Lambda^d,\H(\Lambda^d,\overline{f})^\sharp)}\\
\Gamma^\lambda\ar@{^{(}->}[u] &&&\Gamma^\mu\ar@{^{(}->}[u]\ar[lll]_-{\H(\Lambda^d,\H(\Lambda^d,f)^\sharp)}
}.
$$
By the particular case $S^\lambda=S^\mu=\otimes^d$, we get that $\H(\Lambda^d,\H(\Lambda^d,\overline{f})^\sharp)$ equals $\overline{f}^\sharp$. Thus, $\H(\Lambda^d,\H(\Lambda^d,f)^\sharp)$ equals $f^\sharp$. So the diagram commutes. 

To conclude the proof, we observe that the vertical arrows of the diagram are isomorphisms. So the horizontal maps must be isomorphisms. 
\end{proof}

Now, using lemma \ref{lm-iso}, it is easy to prove the following result.
\begin{theorem}{\cite{Chalupnik2}}
The functor $\Theta$ is an equivalence of categories, with inverse the functor $C\mapsto (\Theta(C^\sharp))^\sharp$.
\end{theorem}
\begin{proof}
To check $\Theta(\Theta(C^\sharp)^\sharp)\simeq C$ we can restrict to complexes of projectives. Then the result follows by lemma \ref{lm-iso}. To check that $\Theta(\Theta(C)^\sharp)^\sharp\simeq C$, we can restrict to complexes of injectives. Once again the result follows by lemma \ref{lm-iso}.
\end{proof}

To give the reader an intuition of the behaviour of $\Theta$, we gather a few explicit computations $\Theta F$ following from lemmas \ref{lm-acyclic} and \ref{lm-comput}.
\begin{example}\label{exemple} We have the following isomorphisms in $\DD^b(\P_{\k})$:
$$\Theta(S^\lambda)\simeq\Lambda^\lambda\;,\quad \Theta(\Lambda^\lambda)\simeq\Gamma^\lambda \;, $$
and $\Theta(\Gamma^2)$ is isomorphic to the cochain complex with $\otimes^2$ in degree $0$, $\Gamma^2$ in degree $1$ and differential $\otimes^2\to \Gamma^2$ given by the multplication (sending $v_1\otimes v_2$ onto $v_1\otimes v_2+v_2\otimes v_1\in \Gamma^2(V)=(V^{\otimes 2})^{\Si_2}$). 

Observe that the degree zero homology of $\Theta(\Gamma^2)$ equals $\Lambda^2$ if $2\ne 0$ in the ground ring $\k$, and $\Gamma^2$ otherwise. Moreover, $\Theta(\Gamma^2)$ has trivial homology in degree $1$ if and only if  $2$ is invertible in the ground ring $\k$.
\end{example}

\section{Derived functors \`a la Dold-Puppe}\label{section-DP}

The derived functors $L_*F(V;n)$ of a non-additive functor $F$ were introduced by Dold and Puppe in \cite{DP1,DP2}. In  section \ref{subsec-rappel-DP}, we recall the classical definition of derived functors. In section \ref{subsec-DP-strict}, we adapt derivation of functors to the framework of strict polynomial functors. 

\subsection{Derivation of ordinary functors}\label{subsec-rappel-DP}

\subsubsection{The Dold-Kan correspondence} Let $\A$ be an abelian category and let $s\A$ denote the category of simplicial objects in $\A$. The Dold-Kan correspondence asserts that the normalized chain functor $\N$ yields an equivalence of categories (with inverse denoted by $\K$):
$$\N:s\A\leftrightarrows \Ch_{\ge 0}(\A):\K \;,$$
and moreover that $\N$ and $\K$ preserve homotopy relations between maps \cite[Section 3]{DP2}, \cite{Do}, \cite[Thm 8.4.1]{Weibel}.

For the homotopy types of complexes, it is equivalent to use the normalized chain functor $\N$ and the unnormalized chain functor $\C$. Indeed, the canonical projection $\C X\to \N X$ is a homotopy equivalence \cite[Satz 3.22]{DP2}, \cite[VIII Thm 6.1]{ML}.

\subsubsection{Derived functors}
Let $F:\A\to \B$ be a (non-necessarily additive) functor between two abelian categories, and assume that $\A$ has enough projectives. Let $V\in \A$ and let $n\ge 0$. Dold and Puppe defined the derived functors $L_qF(V;n)$ (called the $q$-th derived functor of $F$ with height $n$) by the following formula:
$$L_qF(V;n):= H_q\big(\C F\K (P[-n])\big)\;, $$
where $P$ is a projective resolution of $V$ in $\A$, and $[-n]$ refers to the suspension in $\Ch_{\ge 0}(\A)$ (thus $P[-n]_i= P_{i-n}$).

\begin{example}
If $F$ is additive, one proves \cite[4.7]{DP2} that  $L_q F(V;n)$ equals the usual left derived functor $L_{q-n}F(V)$, as defined e.g. in \cite[Chap. 2]{Weibel}.
\end{example}

\begin{remark}
Using the language of homotopical algebra, the definition of Dold and Puppe may be rephrased as follows. Consider the standard model structure on $s\A$ \cite[II.4]{Quillen}. Then $F$ induces a derived functor (in the sense of Quillen) $LF:\mathrm{ho}(s\A)\to \mathrm{ho}(s\B)$. Now $L_q F(V;n)$ is just the $q$-th homotopy group of the value of $L F$ on $\K(V;n)$, where the latter denotes a simplicial object with trivial homotopy groups except for $\pi_n(\K(V;n))=V$.
\end{remark}

\subsubsection{Products and shuffle maps}
Let $\A$ be an additive category, equipped with a symmetric monoidal product $\otimes:\A\times\A\to \A$, which is additive with respect to both variables. Then both $\Ch_{\ge 0}(\A)$ and $s\A$ inherit a symmetric monoidal product. So we can associate two chain complexes to a pair $(X,Y)$ of simplicial objects in $\A$, namely $(\C X) \otimes (\C Y)$ and $\C(X\otimes Y)$. The Eilenberg-Zilber theorem asserts that the shuffle map:
$$\nabla:(\C X) \otimes (\C Y)\to \C(X\otimes Y)$$
is a homotopy equivalence. It is natural with respect to $X,Y$ and associative, symmetric and unital (in other words: the chain functor $\C$ is a lax symmetric monoidal functor), as we can see it from the explicit expression of $\nabla$ \cite[Satz 2.15]{DP2}, \cite[VIII Thm 8.8]{ML}, \cite[Section 8.5.3]{Weibel}.

In particular, if our functors take values in the category of $\k$-modules, the shuffle map induces a morphism (natural in $F,G,V$):
$$ L_pF(V;n)\otimes L_qG(V;n) \to L_{p+q}(F\otimes G)(V;n). $$
Indeed, the left hand side injects (via the K\"unneth morphism) into the homology of the complex $\C F\K(P[-n])\otimes \C G\K(P[-n])$, while the right hand side is the homology of the complex $\C \left( F\K(P[-n]) \otimes G\K(P[-n])\right)$.

\subsection{Derivation of strict polynomial functors}\label{subsec-DP-strict}

We shall denote for short by $K(n)$ the simplicial free $\k$-module $\K(\k[-n])$. If $V$ is a $\k$-module, the definition of the Dold-Kan functor $\K$ yields an isomorphism, natural with respect to $V$:  
$$K(n)\otimes V\simeq \K(V[-n]).\qquad(**)$$

We observe that the definition of derived functors from the preceding section makes sense in the framework of strict polynomial functors. To be more specific, for $F\in\P_{d,\k}$, $q\ge 0$ and $n\ge 0$, we define the $q$-th derived functor of $F$ with height $n$
$$L_qF(-;n):\Gamma^d\V_\k\to \k\text{-mod}$$
as the $q$-th homology group of the complex $\C F_{K(n)}$. 
\begin{remark}\label{rq-1} For strict polynomial functors of degree $d=0$, the above formula still makes sense. A functor of degree zero is just a constant functor with value $F$, so the simplicial $\k$-module $F_{K(n)}$ equals $F$ in each degree, and the face and degeneracy operators are identity maps. So $L_q F(V;n)$ equals zero in positive degrees and $F$ in degree zero.
\end{remark}

Derivation of strict polynomial functors extends the derivation of ordinary functors. Indeed, thanks to isomorphism $(**)$, there is a commutative diagram where the horizontal morphisms are induced by taking the $q$-th derived functors with height $n$ (and $\widetilde{\P}_{d,\k}$ is the category of strict polynomial functors with values in arbitrary $\k$-modules, cf \ref{rk-PT}).
$$\xymatrix{
\P_{d,\k}\ar[r]\ar[d]^-{\U}&\widetilde{\P}_{d,\k}\ar[d]^-{\U}\\
\F_\k\ar[r] &\F_\k
}.$$ 

Finally, we can take the direct sum of the categories $\P_{d,\k}$ and interpret the $q$-th derived functor with height $n$ as a functor $\P_{\k}\to \widetilde{\P}_{\k}$.

\subsection{Derivation in $\DD^-(\P_{\k})$} 

In this section, $\k$ is a PID. We lift the definition of derivation to the level of the derived category $\DD^-(\P_{\k})$.

Let us first introduce the chain functor $\MC$ for mixed bicomplexes.
A \emph{mixed bicomplex} in an additive category $\A$ is just an object of $\Ch^-(s\A)$. If $X$ is a mixed bicomplex in $\A$, we may apply the chain functor $\C$ to each object $X_i$ of $X$. In this way we obtain a bicomplex $(X_{i})_{j}$, whose first differential $d_1^X:(X_i)_j\to (X_{i-1})_j$ is induced by the differential of $X$ and whose second differential $d_2^X:(X_i)_j\to (X_i)_{j-1}$ is induced by the simplicial structure of $X_i$. We denote by $\MC X$ the total complex associated to this bicomplex.

\begin{definition}
Let $n$ be a positive integer and let $C\in\Ch^-(\P_{\k})$. The derived complex of $C$ is the complex $L(C;n)\in \Ch^-(\P_{\k})$ defined by:
$$L(C;n)=\MC(C_{K(n)})\;.$$
Derivation of complexes yields an additive functor:
$$\begin{array}{cccc}
L(-;n): & \Ch^-(\P_{\k}) & \to & \Ch^-(\P_{\k})\\
& C & \mapsto & L(C;n)
\end{array}.$$
\end{definition}

\begin{proposition}\label{prop-DDPd}
The functor $L(-;n)$ induces a triangle functor at the level of the derived categories, still denoted by $L(-;n)$:
$$L(-;n):\DD^-(\P_{\k})\to \DD^-(\P_{\k})\;.$$ 
\end{proposition}
\begin{proof}
This is a standard verification. First, derivation preserves homotopies: if $f:C\to D$ is homotopic to zero via a homotopy $h$, then $L(f;n)$ is homotopic to zero via the homotopy $L(h;n)$. In particular, derivation induces a functor at the level of the homotopy categories $\KK^-(\P_{\k})$. Derivation commutes with suspension, actually we have
an \emph{equality} of complexes $L(C[1];n)= L(C;n)[1]$. Derivation also preserves triangles of $\KK^-(\P_{\k})$. Indeed, if $f$ is a chain map, the chain complex $L(\mathrm{Cone}(f);n)$ is \emph{equal} to $\mathrm{Cone}(L(f;n))$. So $L(-;n)$ induces a triangle endofunctor of $\KK^-(\P_{\k})$. By a standard spectral sequence argument on the bicomplexes $(C_i)_j$, this triangle functor preserves quasi-isomorphisms, whence the result.
\end{proof}

Now we give details about the compatibility of derivation with tensor products. If $X$ and $Y$ are mixed bicomplexes in $\P_\k$, their tensor product $X\ohat Y$ is defined as follows. As a bigraded object we have:
$$(X\ohat Y)_{k,\ell}=\bigoplus_{i+j=k} X_{i,\ell}\otimes Y_{j,\ell}\;.$$
The differential of an element $x\otimes y\in X_{i,\ell}\otimes Y_{j,\ell}$ is given by 
$$\partial(x\otimes y)=\partial_X(x)\otimes y+(-1)^i x\otimes \partial_Y(y)\;,$$
and the simplicial structure is the diagonal one: $d_t=d_t^X\otimes d^Y_t$, $s_t=s_t^X\otimes s_t^Y$.
The tensor product of mixed bicomplexes is a symmetric monoidal product on $\Ch^-(s\P_\k)$. The unit for $\ohat$ is the simplicial object $\K(\k)$ (obtained by applying the Dold-Kan functor $\K$ the constant functor with value $\k$) considered as a complex concentrated in degree zero, and the symmetry isomorphism $\tau:X\ohat Y\to Y\ohat X$ sends $x\otimes y\in X_{i,\ell}\otimes Y_{j,\ell}$ to $(-1)^{ij}y\otimes x$.

There are shuffle maps for mixed bicomplexes in $\P_\k$. To be more specific, we define a morphism of complexes (natural in $X,Y$): 
$$\oshuffle: \MC X \otimes \MC Y\to \MC\,(X\ohat Y)$$
by sending an element $x\otimes y\in (X_i)_k\otimes (Y_j)_\ell$ to $(-1)^{jk}\nabla(x\otimes y)$, where $\nabla$ denotes the usual shuffle map (the sign is needed to get a morphism of complexes). 

If $C$ and $D$ are bounded above cochain complexes of strict polynomial functors, the mixed bicomplex $(C\otimes D)_{K(n)}$ equals $C_{K(n)}\ohat D_{K(n)}$, so the complex $L(C\otimes D;n)$ equals $\MC\,(C_{K(n)}\ohat D_{K(n)})$. Hence, the shuffle map $\oshuffle$ yields a morphism of complexes, natural with respect to $C$ and $D$:
$$L(C;n)\otimes L(D;n)\xrightarrow[]{\oshuffle} L(C\otimes D;n).$$
As already observed in remark \ref{rq-1}, if $\k$ denotes the constant functor with value $\k$, the complex $L(\k;n)$ equals $\k$ in each degree. So if we consider $\k$ as a complex concentrated in degree zero, there is a (unique) quasi-isomorphism 
$$\phi: \k\to L(\k;n)$$ 
which equals the identity map in degree zero.

\begin{proposition}\label{prop-sym}
Derivation with height $n$ induces a lax symmetric monoidal functor at the level of complexes
$$(L(-;n),\oshuffle,\phi)\;:\;(\Ch^-(\P_\k),\otimes,\k)\to (\Ch^-(\P_\k),\otimes,\k),$$
and a symmetric monoidal functor at the level of the derived category:
$$(L(-;n),\oshuffle,\phi)\;:\;(\DD^-(\P_{\k}),\otimes,\k) \to (\DD^-(\P_{\k}),\otimes,\k).$$
\end{proposition}
\begin{proof}
We shall prove slightly more general facts. First, let us denote by $\phi:\C\K(\k)\to \k$ the morphism of chain complexes which is the identity map in degree zero. Then we have a lax monoidal functor: $$(\MC,\oshuffle,\phi):(\Ch^-(s\P_\k),\ohat,\K(\k))\to (\Ch^-(\P_\k),\otimes,\k)\;.$$
Indeed, unity, associativity and commutativity for $\oshuffle$ follow from the fact that the shuffle map $\nabla$ is unital, associative and commutative (not up to homotopy!), as we can see it from the explicit formula of \cite[VIII, Thm 8.8]{ML}. In particular, when we restrict to mixed bicomplexes of the form $C_{K(n)}$, we obtain the first part of proposition \ref{prop-sym}. 

The second part of proposition \ref{prop-sym} follows from the slightly more general fact that for arbitrary mixed bicomplexes $X$ and $Y$, the shuffle morphism $\oshuffle: \MC X \otimes \MC Y\to \MC\,(X\ohat Y)$ is a quasi-isomorphism. To prove this, we write $\oshuffle$ as the composition of two quasi-isomorphisms $\mu$ and $\nu$ defined as follows.
The morphism of quadrigraded objects $(\C X_{i})_k\otimes (\C Y_{j})_\ell\to (\C X_{i})_k\otimes (\C Y_{j})_\ell$ mapping $x\otimes y$ to $(-1)^{jk}x\otimes y$ induces an isomorphism $\nu$ between the chain complex $\MC(X)\otimes\MC(Y)$ and the totalization of the bicomplex 
\begin{align*}&(\C X\otimes\C Y)_{m,n}:=\bigoplus_{i+j=m,k+\ell=n} (\C X_{i})_k\otimes (\C Y_{j})_\ell\;.
\end{align*}
(The differentials of this bicomplex are defined as follows. The restriction of the first differential of $(\C X\otimes\C Y)_{m,n}$ to $(\C X_{i})_k\otimes (\C Y_{j})_\ell$ equals $d_1^X\otimes\Id+(-1)^{i} \Id\otimes d_1^Y$, and the restriction of the second differential equals $d_2^X\otimes\Id+(-1)^k \Id\otimes d_2^Y$.)
There is a morphism of bicomplexes from $(\C X\otimes\C Y)_{m,n}$ to the bicomplex
\begin{align*}
&\C(X\ohat Y)_{m,n}:= \bigoplus_{i+j=m}(\C(X\ohat Y)_{i+j})_{n}\;.
\end{align*}
whose restriction to the $n$-th rows equals $\nabla: (\C X\otimes\C Y)_{*,n}\to \C(X\ohat Y)_{*,n}$. Its totalization  $\mu$ is a quasi-isomorphism. Indeed, $\nabla$ is a quasi-isomorphism (it is even a homotopy equivalence) so we prove this by comparing the spectral sequences associated to the bicomplexes.
Now $\MC(X\ohat Y)$ is the total complex associated to $\C(X\ohat Y)_{m,n}$ and $\oshuffle=\mu\circ\nu$, so $\oshuffle$ is a quasi-isomorphism.
\end{proof}

\subsection{Iterated derivations}
Let $m$ and $n$ be positive integers. Then we may compose derivation with height $m$ and derivation with height $n$ to get a functor
$$ L(-;n)\circ L(-;m): \DD^-(\P_\k)\to \DD^-(\P_\k)\;.$$ 
Since unnormalized derivations are triangle functors and symmetric monoidal functors, so is the composite $L(-;n)\circ L(-;m)$.
The main result about iterated derivations is the following.
\begin{theorem}\label{thm-iter} There is 
an isomorphism of endofunctors of $\DD^-(\P_{\k})$:
$$ L(-;n)\circ L(-;m)\simeq L(-;n+m)\;. $$ 
This isomorphism is an isomorphism of triangle functors as well as an isomorphism of the symmetric monoidal functors. 
\end{theorem}

\begin{corollary}\label{cor-iter}
There is an isomorphism of endofunctors of $\DD^-(\P_{\k})$:
$$L(-;n)\simeq \underbrace{L(-;1)\circ\dots \circ L(-;1)}_{\text{$n$ times}}\;.$$
This isomorphism is an isomorphism of triangle functors as well as an isomorphism of the symmetric monoidal functor.
\end{corollary}

We will actually prove theorem \ref{thm-iter} at the level of chain complexes. To be more specific, let us consider $L(-;m)$ and $L(-;n)$ as endofunctors of the category $\Ch^-(\P_{\k})$. Then theorem \ref{thm-iter} is a consequence of the following stronger statement.

\begin{theorem}\label{thm-iter-strong}
Let $C\in \Ch^-(\P_{d,\k})$ and let $n$ be a positive integer. There is a morphism of complexes, natural with respect to $C$,
$$L(L(C;n);m)\xrightarrow[]{\zeta}  L(C;n+m)\;,$$
satisfying the following properties:
\begin{enumerate}
\item[(1)] $\zeta$ is a homotopy equivalence,
\item[(2)] $\zeta$ commutes with suspension, i.e. there is a commutative diagram:
$$\xymatrix{
L(L(C[1];n);m)\ar[r]^-{\zeta}\ar[d]^-{=}& L(C[1];n+m)\ar[d]^-{=}\\
L(L(C;n);m)[1]\ar[r]^-{\zeta[1]}& L(C;n+m)[1]
}.$$
\item[(3)] If $C\in\Ch^-(\P_{d,\k})$ and $D\in\Ch^-(\P_{e,\k})$, the following diagram commutes up to homotopy:
$$\xymatrix{
L(L(C;n);m)\otimes L(L(D;n);m)\ar[r]^-{\zeta^{\otimes 2}}\ar[d]^-{\oshuffle}& L(C;n+m)\otimes L(D;n+m)\ar[dd]^-{\oshuffle}\\
L(L(C;n)\otimes L(D;n);m)\ar[d]^-{L(\oshuffle;m)}\\
L(L(C\otimes D;n);m)\ar[r]^-{\zeta}& L(C\otimes D;n+m)
}.$$
\end{enumerate}
\end{theorem}

The remainder of the section is devoted to the proof of theorem \ref{thm-iter-strong}. The proof will essentially result from the generalized Eilenberg-Zilber Theorem.
\subsubsection{The generalized Eilenberg-Zilber theorem}
In this paragraph, $\A$ is an additive full subcategory of the category of modules over a ring $R$, containing $R$ as an object. For example, take $\A=\P_{d,\k}$, which is an additive subcategory of the category of modules over the Schur algebra (cf. section \ref{subsubsec-SGR} and appendix \ref{app}). 

An $n$-simplicial object in $\A$ is an $n$-graded object $X_{i_1,\dots,i_n}$ equipped with $n$ simplicial structures (each simplicial structure being relative to one of the indices of $X$), which commute with one another.
We denote by $\C X$ the total chain complex associated to an $n$-simplicial object. It is the total complex of the multicomplex with $n$ differentials obtained by applying the chain functor to each one of the simplicial structures of $X$.

If $\alpha:\{1,\dots,n\}\twoheadrightarrow \{1,\dots,d\}$ is a surjective map, we denote by $\diag^\alpha X$ the associated partial diagonal. That is, $\diag^\alpha X$ is the $d$-simplicial object with 
$$(\diag^\alpha X)_{k_1,\dots,k_d}= X_{k_{\alpha(1)}, k_{\alpha(2)},\dots, k_{\alpha(n)}}$$
and the corresponding diagonal simplicial structure. 

We shall often denote surjective maps $\alpha:\{1,\dots,n\}\twoheadrightarrow \{1,\dots,d\}$ as $d$-tuples where the $i$-th element of the list consists of the elements of the set $\alpha^{-1}\{i\}$, separated by symbols `$+$'. For example the map $\alpha:\{1,2,3\}\to \{1,2\}$ with $\alpha(1)=\alpha(3)=2$ and $\alpha(2)=1$ will be denoted by $(2,1+3)$. With this more suggestive notation, we have:
$$(\diag^{(2,1+3)}X)_{i,j}=X_{j,i,j}\;,$$  
and the first simplicial structure of $\diag^{(2,1+3)}X$ is given by the second simplicial structure of $X$, while the second simplicial structure of $\diag^{(2,1+3)}X$ is given by taking the diagonal simplicial structure of the first and the third simplicial structure of $X$.

The formula of the shuffle map $\nabla$ involved in the classical Eilenberg-Zilber theorem \cite[VIII Thm 8.8]{ML} actually yields a chain map, which is natural with respect to the bisimplicial object $X$, and which equals the identity in degree zero:
$$\nabla:\C X=\C \diag^{(1,2)}X\to \C\diag^{(1+2)}X\;. $$
This example justifies the following terminology.
\begin{definition}Let $\alpha$ and $\beta$ be surjective maps. An Eilenberg-Zilber map is a morphism of complexes, natural with respect to $n$-simplicial objects $X$, and which equals the identity map in degree zero:
$$\C\diag^\alpha(X)\to \C\diag^\beta(X).$$  
\end{definition}

\begin{example}~ 
\begin{enumerate}

\item The shuffle map $\nabla:\C \diag^{(1,2)}X\to \C\diag^{(1+2)}X$ is an Eilenberg-Zilber map.

\item The composite of Eilenberg-Zilber maps is an Eilenberg-Zilber map. 

\item If $\sigma\in\Si_n$ is a permutation of $\{1,\dots,n\}$, there is an Eilenberg-Zilber map:
$$\begin{array}{cccc}
\sigma:& \C X &\to & \C\diag^\sigma X\\
& x &\mapsto & \epsilon(\sigma,x) x
\end{array}$$
where $\epsilon(\sigma,x)$ is a sign defined as follows. If $x$ has $n$-degree $(i_1,\dots,i_n)$, we consider a free graded ring $R$ over generators $e_1,\dots,e_n$ such that each $e_k$ has degree $i_k$. Then $\epsilon(\sigma,x)\in\{\pm 1\}$ is the sign such that
$$e_1\cdot\dots\cdot e_n = \epsilon(\sigma,x)\,e_{\sigma(1)}\cdot\dots\cdot e_{\sigma(n)}\;.$$

\item If $\alpha:\{1,\dots,n\}\twoheadrightarrow \{1,\dots,d\}$ and $\beta:\{1,\dots,m\}\twoheadrightarrow \{1,\dots,e\}$ are surjective maps we denote by $\alpha|\beta$ their concatenation:
$$\begin{array}{cccl}
\alpha|\beta:& \{1,\dots,m+n\} &\twoheadrightarrow & \{1,\dots,d+e\}\\
& i &\mapsto & \left\{\begin{array}{l}\alpha(i)\text{ if $i\le n$,}\\\beta(i-n)+d\text{ if $i> n$.}\end{array}\right.
\end{array}$$
Given two Eilenberg-Zilber maps $f_i:\C\diag^{\alpha_i}(X)\to \C\diag^{\beta_i}(X)$ for $i=1,2$, we may concatenate them to get another Eilenberg-Zilber map:
$$ f_1|f_2: \C\diag^{\alpha_1|\alpha_2}(X)\to \C\diag^{\beta_1|\beta_2}(X)\;.$$

\item If $\gamma:\{1,\dots,m\}\twoheadrightarrow \{1,\dots,n\}$ is surjective and if $f:\C\diag^{\alpha}(X)\to \C\diag^{\beta}(X)$ is an Eilenberg-Zilber map natural with respect to $n$-simplicial objects, we can restrict it to $n$-simplicial objects of the form $\diag^\gamma Y$ to get an Eilenberg-Zilber map
$$f:\C\diag^{\alpha\circ \gamma}(Y)\to \C\diag^{\beta\circ\gamma}(Y)\;. $$
\end{enumerate}
\end{example}

The following result is well-known \cite[Bemerkung 2.16]{DP2}, and it is proved by the methods of acyclic models.

\begin{proposition}[Generalized Eilenberg-Zilber theorem]\label{prop-general-EZ} Let $X$ be an $n$-simplicial object in the additive category $\A$, and let  $\alpha:\{1,\dots,n\}\twoheadrightarrow \{1,\dots,d\}$ and $\beta:\{1,\dots,n\}\twoheadrightarrow \{1,\dots,e\}$ be surjective maps.  
There exists an Eilenberg-Zilber map:
$$\C\diag^\alpha(X)\to \C\diag^\beta(X),$$
and two such Eilenberg-Zilber maps are chain homotopic, via a homotopy which is natural with respect to $X$. 
\end{proposition}
\begin{proof}[Sketch of proof.]
First, an $n$-simplicial object in $\A$ is a functor $(\Delta^{\mathrm{op}})^{\times n}\to \A$. Since $\A$ is a subcategory of $R\text{-mod}$ containing the free $R$-modules of finite rank, $n$-simplicial objects in $\A$ include the acyclic models $M_I:=R(\Delta[i_1]\times\dots\times\Delta[i_n])$ for all $n$-tuples $I=(i_1,\dots,i_n)$ of nonnegative integers (i.e. $\Delta[i_1]\times\dots\times\Delta[i_n]$ is the $n$-simplicial set $\prod\hom_{\Delta}(-,[i_k])$, and $M_I$ is the $n$-simplicial $R$-module obtained by taking the free $R$-module over it). Moreover, by the Yoneda lemma, homomorphisms of $n$-simplicial objects from such an acyclic model $M_I$ to an $n$-simplicial object $X$ is determined by an isomorphism: $\hom(M_I,X)\simeq X_I$. 

Hence, for all $x\in X_I$ there exists a unique map $\phi:M_I\to X$ such that $x=\phi(\iota_I)$ where $\iota_I$ is the characteristic simplex of $\Delta[i_1]\times\dots\times\Delta[i_n]$, i.e. the identity map of $[i_1]\times\dots \times [i_n]\in (\Delta^{\mathrm{op}})^{\times n}$.
So, if $t$ is an Eilenberg-Zilber map (or a homotopy between Eilenberg-Zilber maps), the commutative diagram:
$$\xymatrix{
\C\diag^\alpha(M_I)\ar[r]^-{t}\ar[d]^-{C\diag^\alpha(\phi)}& \C\diag^\beta(M_I)\ar[d]^-{C\diag^\beta(\phi)}\\
\C\diag^\alpha(X)\ar[r]^-{t}& \C\diag^\beta(X)
}$$
shows that the value of $t$ on $x\in X_I$ is completely determined by the restriction of $t$ to the acyclic model $M_I$. Hence, $t$ is completely determined by its restriction to the full subcategory of $\A$ whose objects are the acyclic models $M_I$. Conversely, if $t$ is defined on the full subcategory of $\A$ whose objects are the acyclic models $M_I$, the commutative diagram (and the uniqueness of $\phi$) show that $t$ may be extended uniquely to $\A$.

So it suffices to prove the proposition when $X$ is an acyclic model. In that case, the result is standard, and proved exactly as in the proof of \cite[VIII, lemmas 8.2 and 8.3]{ML}.
\end{proof}

\begin{corollary}
Eilenberg-Zilber maps are homotopy equivalences. 
\end{corollary}

\subsubsection{Proof of theorem \ref{thm-iter-strong}}
The proof consists of several steps.

\medskip

{\bf Step 1. A replacement for $L(C;n+m)$.} We denote by $K(n)\boxtimes K(m)$ the bisimplicial $\k$-module obtained as tensor product of $K(n)$ and $K(m)$, and by $K(n,m)$ its diagonal. We consider the lax monoidal functor
$$(\L(-;n+m),\oshuffle,\phi):(\Ch^-(\P_\k),\otimes,\k)\to (\Ch^-(\P_\k),\otimes,\k)$$
where $\L(C;n+m)=\MC \left( C_{K(n,m)}\right)$, and $\oshuffle$ is the shuffle map for mixed complexes.

By the Eilenberg-Zilber theorem and the K\"unneth theorem, the homology of $\C K(n,m)$ equals $\k$ concentrated in homological degree $n+m$. Hence, the Dold-Kan correspondence yields a homotopy equivalence of simplicial $\k$-modules $K(n,m)\simeq K(n+m)$. This homotopy equivalence induces a homotopy equivalence of chain complexes:
$$\L(C;n+m)=\MC\left( C_{K(n,m)}\right)\to \MC\left( C_{K(n+m)}\right)=L(C;n+m)\;,$$
natural with respect to $C$, commuting with suspension and with $\oshuffle$.
Hence we can replace $L(C;n+m)$ by $\L(C;n+m)$ in the proof.

\medskip

{\bf Step 2. Construction of $\zeta$.} 
Let $C$ be a chain complex. We first observe that we have equalities of complexes:
$$L(L(C;n);m)= \Tot\left(\Tot^{(1+2,3)}C''\right) = \Tot C'' = \Tot\left(\Tot^{(1,2+3)} C''\right)\;,$$
In these equalities, $C''$ denotes the tricomplex $(C_i)_{K(n)_j\otimes K(m)_k}$, with first differential induced by the differential of $C$, and with second (resp. third) differential induced by the simplicial structure of $K(n)$ (resp. $K(m)$), and $\Tot^{(1,2+3)}C''$ is the bicomplex obtained by totalizing the second and third differential of $C''$. 

For each $i$, the simplicial object $(C_i)_{K(n,m)}$ is the diagonal of the bisimplicial object $(C_i)_{K(n)\otimes K(m)}$. So, proposition \ref{prop-general-EZ} yields an Eilenberg-Zilber map:
$$f: \C \left((C_i)_{K(n)\otimes K(m)}\right)\to \C \left( (C_i)_{K(n,m)}\right)\;.$$
natural with respect to $C_i$, hence a morphism of bicomplexes, natural with respect to $C$:
$$f:\Tot^{(1,2+3)} C''\to C'\;.$$
We define $\zeta$ as the morphism of chain complexes  (natural with respect to $C$) induced at the level of total complexes:
$$\zeta: L(L(C;n);m)= \Tot\Tot^{(1,2+3)} C''\to \Tot C'=\L(C;m+n)\;.$$

We readily check from the definition of $\zeta$ that condition (2) is satisfied. So it remains to check conditions (1) and (3).

\medskip

{\bf Step 3. Reduction to nonnegative complexes.} Since condition (2) is satisfied, we see that $\zeta$ is a homotopy equivalence for $C[1]$ if and only if it is a homotopy equivalence for $C$. Hence, it suffices to prove condition (1) with $C\in\Ch_{\ge 0}(\P_{d,\k})$. Similarly the diagram involved in condition (3) commutes up to homotopy for given $C$ and $D$ if and only if it commutes up to homotopy for their suspensions $C[2i]$ and $D[2j]$ (taking even suspensions enables us to avoid checking signs). Hence it suffices to prove condition (3) with $C\in\Ch_{\ge 0}(\P_{d,\k})$ and $D\in \Ch_{\ge 0}(\P_{e,\k})$.

\medskip

{\bf Step 4. Reduction to the simplicial case.} For $C\in\Ch_{\ge 0}(\P_{d,\k})$, the Dold-Kan correspondence
$$\N:s(\widetilde{\P}_{d,\k})\leftrightarrows \Ch^-(\widetilde{\P}_{d,\k}):\K$$
ensures that $\C\K(C)$ is homotopy equivalent to $C$. Moreover the explicit formula for $\K$ \cite[Section 8.4]{Weibel} ensures that $\C\K(C)$ takes finitely generated projective values, i.e. is an element of $\Ch_{\ge 0}(\P_{d,\k})$. Since derivation of functors preserves homotopy equivalences, conditions (1) and (3) hold for $C$ (and $D$) if and only if they hold for the complexes $\C\K(C)$ (and $\C\K(D)$). Therefore, to prove conditions (1) and (3), it suffices to treat the case where there are $X\in s(\P_{d,\k})$ and $Y\in s(\P_{e,\k})$ such that $C=\C X$ and $D=\C Y$.

\medskip 

{\bf Step 5. Proof of condition (1): simplicial case.} If $C=\C X$ with $X\in s(\P_{d,\k})$, then $\zeta:L(L(C;n,m))\to \L(C,n+m)$ is the evaluation of the Eilenberg-Zilber map $\Id|f: \C Z \to \C \diag^{(1,2+3)}Z$ on the $3$-simplicial object $(X_i)_{K(n)_j\otimes K(m)_k}$. Hence it is a homotopy equivalence.

\medskip

{\bf Step 6. Proof of condition (3): simplicial case.} We assume that $C=\C X$ and $D=\C Y$ with 
$X\in s(\P_{d,\k})$ and $Y\in s(\P_{e,\k})$. In that case, we may rewrite the various morphisms in the diagram of condition (3), hence their compositions, as Eilenberg-Zilber maps, so commutation up to homotopy follows from proposition \ref{prop-general-EZ}.
To be more specific, let us denote by $Z$ the $6$-simplicial object 
$$Z_{i,j,k,\ell,r,s}=(X_i)_{K(n)_j\otimes K(m)_k}\otimes (Y_\ell)_{K(n)_r\otimes K(m)_s}\;.$$
Then $\zeta\otimes\zeta$ is the Eilenberg-Zilber map
$$\C Z\xrightarrow[]{\Id|f|\Id|f}\C\diag^{(1,2+3,4,5+6)}Z\;,$$
the map $\oshuffle$ on the right hand side is the Eilenberg-Zilber map given as the composite (where $\sigma\in\Si_4$ exchanges $2$ and $3$), 
$$\C\diag^{(1,2+3,4,5+6)}Z\xrightarrow[]{\sigma} \C\diag^{(1,4,2+3,5+6)}Z\xrightarrow[]{\Id|\Id|\nabla}\C\diag^{(1,4,2+3+5+6)}Z\;,$$
the map $\oshuffle$ on the left hand side is the composite (where $\sigma\in\Si_6$ satisfies $\sigma(3)=4$, $\sigma(4)=5$ and $\sigma(5)=3$):
$$\C Z\xrightarrow[]{\sigma} \C\diag^{(1,2,4,5,3,6)}Z\xrightarrow[]{\Id|\Id|\Id|\Id|\nabla}\C\diag^{(1,2,4,5,3+6)}Z\;,$$
the map $L(\oshuffle;m)$ is the composite (where $\sigma\in \Si_5$ exchanges $2$ and $3$)
$$\C\diag^{(1,2,4,5,3+6)}Z \xrightarrow[]{\sigma} \C\diag^{(1,4,2,5,3+6)}Z\xrightarrow[]{\Id|\Id|\nabla|\Id}\C\diag^{(1,4,2+5,3+6)}Z\;,$$
and the map $\zeta$ is the Eilenberg-Zilber map
$$\C\diag^{(1,4,2+5,3+6)}Z\xrightarrow[]{\Id|\Id|f}\C\diag^{(1,4,2+5+3+6)}Z\;.$$
This finishes the proof of theorem \ref{thm-iter-strong}.

\section{Main theorem}\label{section-main}

\subsection{The functor $\Sigma$}
For all $d\ge 0$ and all $C\in \Ch^b(\P_{d,\k})$, we denote by $\Sigma C$ the shifted complex $C[-d]$ (recall the convention $C[-d]^i=C^{i+d}$). Gathering all $d$ together, we get a functor: 
$$\Sigma:\Ch^b(\P_{\k})\to \Ch^b(\P_{\k})$$ 
For $C\in \DD^b(\P_{d,\k})$ and $D\in \DD^b(\P_{e,\k})$, we denote by $\xi_{d,e}$ the isomorphism of complexes (the sign ensures that $\xi_{d,e}$ commutes with the differentials)
$$\begin{array}{cccc}
\xi_{d,e}:& (C[-d])_i\otimes (D[-e])_j & \to & (C\otimes D)[-d-e]_{i+j}\\
& x\otimes y &\mapsto & (-1)^{e(i+d)} x\otimes y 
\end{array}
$$ 
Gathering all indices $d,e$ together, we get a natural isomorphism:
$$\xi: \Sigma(C)\otimes \Sigma(D)\xrightarrow[]{\,\simeq\,}\Sigma(C\otimes D)\;.$$
The following statement is straightforward from the definitions.
\begin{proposition}
The triple $(\Sigma,\xi,\Id)$ is a monoidal endofunctor of the symmetric monoidal category $(\Ch^b(\P_{\k}),\otimes,\k)$. It induces a triangle monoidal endofunctor, still denoted by $(\Sigma,\xi,\Id)$, of the symmetric monoidal category $(\DD^b(\P_{\k}),\otimes,\k)$
\end{proposition}

\begin{remark}\label{rk-sgn-Sigma}
The monoidal functor $(\Sigma,\xi,\Id)$ is not \emph{symmetric}. Indeed, if $C\in \Ch^b(\P_{d,\k})$ and $D\in \Ch^b(\P_{e,\k})$, the following diagram commutes up to a $(-1)^{de}$ sign:
$$\xymatrix{
(\Sigma C)\otimes (\Sigma D)\ar[r]^-{\xi}\ar[d]^-{\tau} & \Sigma(C\otimes D)\ar[d]^-{\Sigma(\tau)}\\
(\Sigma D)\otimes (\Sigma C)\ar[r]^-{\xi}& \Sigma(D\otimes C)\;.
}$$
\end{remark}

The functor $(\Sigma,\xi,\Id)$ is actually a monoidal \emph{automorphism} of the category of the symmetric monoidal category $(\Ch^b(\P_{\k}),\otimes,\k)$. That is, there is a monoidal functor $(\Sigma^{-1},\xi^{-1},\Id)$ such that the composite of these two monoidal functors equals the monoidal functor $(\Id,\Id,\Id)$. To be more specific, the monoidal inverse sends $C\in\Ch^b(\P_{d,\k})$ to $\Sigma^{-1}(C)= C[d]$, and the restriction of the map $\xi^{-1}$ to complexes of homogeneous functors of degrees $C$ and $D$ is the map
$$\begin{array}{cccc}
(\xi^{-1})_{d,e}: &C[d]_i\otimes D[e]_j&\to& (C\otimes D)[d+e]_{i+j}\;. \\
& x\otimes y & \mapsto &(-1)^{ei} x\otimes y
\end{array}$$
Observe that the sign appearing in the definition of $(\xi^{-1})_{d,e}$ is different from the sign appearing in the definition of $\xi_{d,e}$.
\begin{remark}
Other signs would be possible in the definition of $\xi_{d,e}$ (hence in the definition of $\xi^{-1}_{d,e}$). Our choice of signs is the good one to make theorem \ref{thm-main} work.
\end{remark}

Since $\Sigma$ is defined from the suspension functors, and since triangle functors commute with suspension, the following result holds.
\begin{lemma}\label{lm-commut}
Let $F:\DD^b(\P_\k)\to \DD^b(\P_\k)$ be a triangle functor. The natural transformation $F(C [-1])\simeq F(C)[-1]$ induces an isomorphism of triangle functors:
$$\Sigma\circ F\simeq F\circ \Sigma\;.$$ 
Moreover, if $F=\Theta$, the isomorphism above is also an isomorphism of monoidal functors.
\end{lemma}

Finally, it follows from remarks \ref{rk-sgn-Sigma} and \ref{rk-sgn-Ringel} that the functor $\Sigma\circ \Theta$ is \emph{symmetric} monoidal (although neither $\Sigma$ nor $\Theta$ is). 

\subsection{Main theorem and consequences}

\begin{theorem}\label{thm-main}
Let $\k$ be a PID, and let $n$ be a positive integer. The derivation functor 
$L(-;n):\DD^-(\P_\k)\to \DD^-(\P_\k)$
induces a (triangle, symmetric monoidal) functor 
$$L(-;n):\DD^b(\P_\k)\to \DD^b(\P_\k)\;,$$
and there is an isomorphism of functors 
$$ L(-;n)\simeq  \Sigma^n\circ\Theta^n\;.$$
Moreover, this isomorphism is an isomorphism of triangle functors as well as an isomorphism of monoidal functors.
\end{theorem}

Let us make clear what the first part of the statement means. 
The canonical monoidal functor (induced by the inclusion)
$$\DD^b(\P_\k)\to \DD^-(\P_\k) $$  
is fully faithful, and induces an equivalence of monoidal categories between $\DD^b(\P_\k)$ and the full subcategory $(\DD^b)'$ of $\DD^-(\P_\k)$ whose objects are complexes satisfying $H^n(C)=0$ for $n\ll 0$ (see e.g. \cite[Lemma 11.7]{Keller} for the equivalence of categories, the fact that all the functors involved are monoidal is  a straightforward verification). The first part of theorem \ref{thm-main} means that there exist a triangle monoidal functor (the dashed arrow) making the following diagram commute:
$$\xymatrix{
\DD^-(\P_\k)\ar[r]^-{L(-;n)}&\DD^-(\P_\k)\\
(\DD^b)'\ar@{^{(}->}[u]\ar@{-->}[r]& (\DD^b)'\ar@{^{(}->}[u]
}$$
and we still denote by $L(-;n)$ the composite
$$\DD^b(\P_\k)\simeq  (\DD^b)'\to (\DD^b)'\simeq \DD^b(\P_\k)\;. $$
\begin{remark}
The first part of theorem \ref{thm-main} is actually equivalent to saying that strict polynomial functors have bounded derivatives. So it may be seen as a strict polynomial version of \cite[Satz 4.22]{DP2}. 
\end{remark}

Now let us spell out some consequences of theorem \ref{thm-main}. Restricting to functors and taking homology, we obtain the statement alluded to in the introduction.
\begin{corollary}\label{cor-1}
Let $F\in\P_{d,\k}$ and let $V\in\V_\k$ be a free $\k$-module of finite rank. There is an isomorphism, natural with respect to $V$ and $F$: 
$$H^{i}(\Theta^n F)(V)\simeq L_{nd-i}F(V;n)\;.$$
Moreover, if $F\in\P_{d,\k}$ and $G\in\P_{e,\k}$, these isomorphisms fit into a diagram which commutes up to a $(-1)^{nei}$ sign (and where $\kappa$ denotes the usual K\"unneth morphism \cite[V (10.1)]{ML}):
$$
\xymatrix{
H^{i}(\Theta^n F)(V)\otimes H^j(\Theta^n G)(V)\ar[d]^-{\kappa}\ar[rr]^-{\simeq} && L_{nd-i}F(V;n)\otimes L_{ne-j}G(V;n)\ar[dd]^-{\nabla}\\
H^{i+j}(\Theta^n F\otimes \Theta^n G)(V)\ar[d]^-{H^{i+j}(\lax^n)}&&\\
H^{i+j}(\Theta^n(F\otimes G))(V)\ar[rr]^-{\simeq}&&  L_{nd+ne-i-j}(F\otimes G)(V;n)
}.
$$
\end{corollary}

As it has been said in the introduction, corollary \ref{cor-1} has concrete interpretations in terms of extension groups. If $F,G\in\P_{d,\k}$ and if $V$ is a finitely generated projective $\k$-module $V$, we denote by $\E(F,G)(V)$ the parameterized extension groups:
$$\E^*(F,G)(V):= \Ext^*_{\P_{d,\k}}(F^V,G)\;.$$
Thus, $\E(F,G)$ is a functor from $\Gamma^d\V_\k$ to graded $\k$-modules and $\E^*(F,G)(\k)$ equals $\Ext^*_{\P_{d,\k}}(F,G)$. By definition, the homology of $\Theta F$ equals $\E^*(\Lambda^d,F)$. Moreover, after taking homology, the morphism $\lax$ is nothing but the usual convolution product of extensions (used e.g. in \cite{FFSS, Chalupnik2, TouzeClassical, TouzeEML}):
$$\E^i(\Lambda^{d},F)\otimes\E^j(\Lambda^{e},G)\to \E^{i+j}(\Lambda^{d+e},F\otimes G)\;.$$
Thus, corollary \ref{cor-1} may be reinterpreted in the following way.

\begin{corollary}
Let $\k$ be a PID, let $F\in\P_{d,\k}$ and let $V\in\V_\k$. There are isomorphisms (natural in $F,V$):
$$\E^i(\Lambda^d,F)(V)\simeq L_{d-i}F(V;1)\;.$$
Moreover, for $F\in\P_{d,\k}$ and $G\in\P_{e,\k}$, the pairing 
$$L_{d-i}F(V;1)\otimes L_{e-j}G(V;1)\to L_{d+e-i-j}(F\otimes G)(V;1)$$
identifies through this isomorphism, up to a $(-1)^{ie}$ sign, with the pairing
$$\E^i(\Lambda^{d},F)(V)\otimes\E^j(\Lambda^{e},G)(V)\to \E^{i+j}(\Lambda^{d+e},F\otimes G)(V)\;.$$
\end{corollary}

Similarly, the $2$-fold iteration of Ringel duality has an interpretation in terms of extension groups. Indeed, 
$\Theta^2 F$ equals $\RR\H(\Lambda^d,\Theta F)$, which is isomorphic to $\RR\H(\Theta^{-1}\Lambda^d,F)$ since $\Theta$ is an equivalence of categories. Now, $\Theta^{-1}\Lambda^d=S^d$, 
so corollary \ref{cor-1} yields an isomorphism:
$$L_{2d-i}F(V;2)\simeq \E^i(S^d,F)(V)\;.$$

\subsection{Proof of theorem \ref{thm-main}}\label{subsec-proof-main}

Lemma \ref{lm-commut} yields an isomorphism of triangle monoidal functors $\Sigma^n\circ \Theta^n\simeq (\Sigma\circ \Theta)^n$ and similarly, corollary \ref{cor-iter} yields an isomorphism of triangle monoidal functors $L(-;n)\simeq L(-;1)^n$. So it suffices to prove theorem \ref{thm-main} for $n=1$.  In the latter case, theorem \ref{thm-main} is a consequence of the following statement. 
\begin{proposition}\label{prop-main}
The following two composites are isomorphic as triangle and monoidal functors.
\begin{align*} &\DD^b(\P_{\k})\xrightarrow[]{\Theta} \DD^b(\P_\k)\xrightarrow[]{\Sigma} \DD^b(\P_\k)\to \DD^-(\P_\k)\\
&\DD^b(\P_\k)\to \DD^-(\P_\k)\xrightarrow[]{L(-;1)} \DD^-(\P_\k).\end{align*}
\end{proposition}
The remainder of section \ref{subsec-proof-main} is devoted to the proof of proposition \ref{prop-main}.

\subsubsection{Conventions for $\H(C,D)$}\label{subsubsec-convention} For the proof, we need to consider $\H(C,D)$ when both $C$ and $D$ are complexes. In this paragraph, we give sign conventions for $\H(C,D)$ and their consequences.

If $C\in \Ch^+(\P_\k)$ and $D\in \Ch^-(\P_\k)$, we define $\H(C,D)\in \Ch^-(\P_\k)$ as the total complex of the bicomplex $(\H(C_i,D^j),\H(\partial_C,D^j),\H(C_i,\partial_D))$. Thus,
\begin{equation*}
\begin{aligned}&\H(C,D)^n=\bigoplus_{i+j=n}\H(C_i,D^j)\;,\\
&\partial(f)=f\circ \partial_C + (-1)^{i}\partial_D\circ f\quad \text{ if $f\in \H(C_i,D^j)$.} 
\end{aligned}
\end{equation*}
With this convention, we have the following compatibility results with the tensor products, suspension and duality.

\begin{description}
\item[Tensor products.] Tensor products yield a morphism of complexes
$$\begin{array}{cccc}
\otens: & \H(C_i,D^j)\otimes \H(E_k,F^\ell) &\to & \H(C_i\otimes E_k, D^j,F^\ell)\;.\\
& f\otimes g &\mapsto & (-1)^{jk}f\otimes g    
\end{array}$$
\item[Suspension.] There is an isomorphism of complexes
$$\psi_d: \H(C,D[d])\xrightarrow[]{\simeq} \H(C,D)[d]\;,$$
which sends $f\in \H(C_i,D[d]^j)$ to $\psi(f)=(-1)^{di} f$. Moreover, the following diagram commutes
$$\xymatrix{
\H(C,D[s])\otimes \H(E,F[t]) \ar[rr]^-{\psi_s\otimes\psi_t}\ar[d]^-{\otens}&& \H(C,D)[s]\otimes \H(E,F)[t]\ar[d]^-{\xi_{s,t}}\\
\H(C\otimes E,D[s]\otimes F[t])\ar[d]^-{\H(C\otimes E,\xi_{s,t})}&& \left(\H(C,D)\otimes \H(E,F)\right)[s+t]\ar[d]^-{\otens[s+t]}\\
\H(C\otimes E,(D\otimes F)[s+t])\ar[rr]^-{\psi_{s+t}}&&\H(C\otimes E,D\otimes F)[s+t]\;.
}$$
\item[Duality.] The dual of a complex $C$ is the complex $C^\sharp$ with $(C^\sharp)_i=(C^i)^\sharp$ and $\partial_{C^\sharp}= (\partial_C)^\sharp$ (no sign on the differential of the dual). Duality commutes with tensor products: $C^\sharp\otimes D^\sharp=(C\otimes D)^\sharp$. There is an isomorphism of complexes:
$$\begin{array}{cccc}
^\sharp:&\H((C^{\sharp})_i, D^j) & \xrightarrow[]{\simeq} & \H((D^{\sharp})_j, C^i)\;.\\
&f & \mapsto & (-1)^{ij} f^\sharp  
\end{array}
$$
Moreover, the following diagram commutes:
$$\xymatrix{
\H(C^\sharp, D)\otimes \H(E^\sharp, F) \ar[d]^-{\otens}\ar[r]^-{^\sharp\otimes^\sharp}& \H(D^\sharp, C)\otimes \H(F^\sharp, E)\ar[d]^-{\otens}\\
\H(C^\sharp\otimes E^\sharp, D\otimes F)\ar[r]^-{^\sharp}&\H(D^\sharp\otimes F^\sharp, C\otimes E)\;.
}$$
\end{description}

\subsubsection{Plan of the proof of proposition \ref{prop-main}.} We first give an equivalent definition of $\Theta$.
By duality $\Theta:\DD^b(\P_{d,\k})\to \DD^b(\P_{d,\k})$ is isomorphic to the right derived functor of $F\mapsto \H(F^\sharp,\Lambda^d)$, and the morphism $\lax$ may be described as the composite (where the last map is induced by the multiplication $\Lambda^d\otimes\Lambda^e\to \Lambda^{d+e}$).
$$\RR\H(C^\sharp,\Lambda^d)\otimes \RR\H(D^\sharp,\Lambda^e)\xrightarrow[]{\otimes}\RR\H(C^\sharp\otimes D^\sharp,\Lambda^d\otimes\Lambda^e)\to \RR\H((C\otimes D)^\sharp, \Lambda^{d+e})\;. $$

Now assume that there is a family of complexes of injectives $J(d)\in \Ch^-(\P_{d,\k})$, together with quasi-isomorphisms $\phi_d:\Lambda^d\to J(d)$ and with morphisms $f_{d,e}: J(d)\otimes J(e)\to J(d+e)$ such that the following diagram commutes in $\Ch^-(\P_\k)$:
$$\xymatrix{\Lambda^d\otimes \Lambda^e\ar[d]^-{\phi_d\otimes\phi_e}\ar[r]^-{\mathrm{mult}}&\Lambda^{d+e}\ar[d]^-{\phi_{d+e}}\\
J(d)\otimes J(e)\ar[r]^-{f_{d,e}}& J(d+e)\;.}$$
Then $\Theta$ is isomorphic to the localization of the functor   
$$\begin{array}{cccc}
\wtheta: &\Ch^b(\P_{d,\k})&\to &\Ch^-(\P_{d,\k})\;.\\
& C &\mapsto &\H(C^\sharp,J(d))
\end{array}$$
Moreover, for $C\in \Ch^-(\P_{d,\k})$ and $D\in \Ch^-(\P_{e,\k})$, the morphism $\lax$ corresponds to the natural transformation $\wlax$ defined as the composite (where the last map is induced by $f_{d,e}$):
$$\H(C^\sharp,J(d))\otimes\H(D^\sharp,J(e))\xrightarrow[]{\otens} \H(C^\sharp\otimes D^\sharp, J(d)\otimes J(e))\xrightarrow[]{} \H((C\otimes D)^\sharp, J(d+e))\;,$$
and the unit $\phi: \k\to \Theta(\k)$ is induced by the morphism $\widetilde{\phi}$
$$\k=\H(\k,\k)=\H(\k,\Lambda^0)\to \H(\k,J(0))\;. $$

The proof of proposition \ref{prop-main} is organized as follows.
\begin{itemize}
\item In a first step, we make explicit choices of coresolutions $J(d)$ and maps $f_{d,e}$. Thus we get an explicit monoidal functor $(\wtheta,\wlax,\widetilde{\phi})$. 
\item In a second step, we prove that with our choices, the functors $\Sigma\circ \wtheta$ and $L(-;1)$ are isomorphic as functors from $\Ch^b(\P_{d,\k})$ to $\Ch^-(\P_{d,\k})$, and that this isomorphism commutes with suspension and monoidal structures.
\end{itemize}

\subsubsection{Construction of $(\wtheta,\wlax,\widetilde{\phi})$.}

To define $(\wtheta,\wlax,\widetilde{\phi})$, we take the complexes $J(d):= (\C S^d_{K(1)})[d]$ and the maps $f_{d,e}$ provided by the following lemma.

\begin{lemma}
For all $d\ge 0$, there is a quasi-isomorphism of complexes:
$$\Lambda^d\xrightarrow[]{\phi_d} (\C S^d_{K(1)})[d]\;. $$ 
Moreover, if $\nabla$ denotes the shuffle map, let $f_{d,e}$ denote the composite:
\begin{align*}(\C S^d_{K(1)})[d]&\otimes (\C S^e_{K(1)})[e]\xrightarrow[]{(\xi^{-1})_{d,e}}(\C S^d_{K(1)})\otimes (\C S^e_{K(1)})[d+e]\\&\xrightarrow[]{\nabla[d+e]} \C (S^d\otimes S^e)_{K(1)}[d+e]\xrightarrow[]{\mathrm{mult}[d+e]} \C (S^{d+e})_{K(1)}[d+e]\;,\end{align*}
then the following diagram is commutative: 
$$\xymatrix{\Lambda^d\otimes \Lambda^e\ar[d]^-{\phi_d\otimes\phi_e}\ar[r]^-{\mathrm{mult}}&\Lambda^{d+e}\ar[d]^-{\phi_{d+e}}\\
(\C S^d_{K(1)})[d]\otimes (\C S^e_{K(1)})[e]\ar[r]^-{f_{d,e}}& (\C S^{d+e}_{K(1)})[d+e]\;.}$$
\end{lemma}
\begin{proof}
Since the shuffle map $\nabla$ is strictly associative and strictly commutative, the composite
$$\C S^d_{K(1)}\otimes \C S^e_{K(1)}\xrightarrow[]{\nabla}\C (S^d\otimes S^e)_{K(1)}\xrightarrow[]{\mathrm{mult}} \C S^{d+e}_{K(1)}  $$
makes $\C S_{K(1)}=\bigoplus_{d\ge 0} \C S^d_{K(1)}$ a commutative monoid in $\Ch^-(\P_\k)$.

By definition $\C S^1_{K(1)}= (\C K(1))\otimes S^1$ has homology equal to the functor $S^1=\Lambda^1$, placed in degree $1$. Choosing a cycle in the complex of $\k$-modules $\C K(1)$ and tensoring by $\Lambda^1=S^1$, we obtain a quasi-isomorphism of complexes of strict polynomial functors (put the trivial differential on the left hand side):
$$\Lambda^1[-1]\hookrightarrow \C S^1_{K(1)}\;.$$
Since $\C S_{K(1)}$ is graded commutative, the universal property of exterior algebras yield a morphism of commutative monoids in  $\Ch^-(\P_\k)$:
$$A=\bigoplus_{d\ge 0}\Lambda^d[-d]\hookrightarrow \bigoplus_{d\ge 0} \C S^d_{K(1)}=\C S_{K(1)}\;. \qquad(\dag)$$

It is well known that $(\dag)$ is a quasi-isomorphism. Indeed, for all free finitely generated $\k$-modules $U$, $V$, there is a commutative diagram (where the horizontal quasi-isomorphisms are defined via multiplications):
$$\xymatrix{
A(U)\otimes A(V)\ar@{^{(}->}[d]\ar[rr]^-{\simeq}&& A(U\oplus V)\ar@{^{(}->}[d]\\
\C S_{K(1)}(U)\otimes \C S_{K(1)}(V)\ar[rr]^-{\mathrm{mult}\,\circ\nabla}&&\C S_{K(1)}(U\oplus V)
}$$ 
so it suffices to prove that $\k[-1]=A(\k)\hookrightarrow \C S_{K(1)}(\k)$ is a quasi-isomorphism. By construction the map $\k[-1]\hookrightarrow S^1_{K(1)}(\k)=K(1)$ is a quasi isomorphism, so we only have to check that the homology of $\bigoplus_{d\ne 1}\C S^d_{K(1)}(\k)$ vanishes. We readily check that $\N S(K(1))=\overline{B} (S(\k))$ (the bar complex of the symmetric algebra on one generator), so the vanishing follows from the equalities (the last equality follows e.g. from \cite[VII Thm 2.2]{ML})
$$H(\C S_{K(1)}(\k))\simeq H(\N S_{K(1)}(\k))= H(\overline{B}(S(\k)))=\mathrm{Tor}^{S(k)}(\k,\k)=\k[-1]\;.$$ 

Restricting the quasi-isomorphism $(\dag)$ to the homogeneous part of degree $d$ and shifting by $d$ we get the required isomorphism:
$$\phi_d:\Lambda^d=(\Lambda^d[-d])[d]\hookrightarrow (\C S^d_{K(1)})[d]\;.$$
Moreover, since $(\dag)$ is a morphism of monoids, we have commutative diagrams:
$$
\xymatrix{(\Lambda^d[-d])[d]\otimes (\Lambda^e[-e])[e]\ar[d]^-{\phi_d\otimes\phi_e}\ar[rrr]^-{\mathrm{mult}[d+e]\circ(\xi^{-1})_{d,e}} &&&\Lambda^{d+e}[-d-e][d+e]\ar[d]^-{\phi_{d+e}}\\
(\C S^d_{K(1)})[d]\otimes (\C S^e_{K(1)})[e]\ar[rrr]^-{f_{d,e}}&&& (\C S^{d+e}_{K(1)})[d+e]}.
$$
To finish the proof, we observe that the sign induced by $(\xi^{-1})_{d,e}$ is equal to $1$. Hence the upper horizontal arrow identifies with the multiplication $\Lambda^d\otimes\Lambda^e\to \Lambda^{d+e}$.
\end{proof}

\subsubsection{Proof of the isomorphism $\Sigma\circ \wtheta\simeq L(-;1)$.} To prove the isomorphism between $\Sigma\circ \wtheta$  and $L(-;1)$ commuting with suspension and monoidal structures, we first introduce yet another monoidal functor 
$$L':\Ch^b(\P_\k)\to \Ch^-(\P_\k)\;.$$ 
Namely, if $C\in\Ch^b(\P_{d,\k})$, we let 
$$L'(C)=\H(C^\sharp, \C S^d_{K(1)})\;.$$
We observe that $L'(C[1])$ equals $L'(C)[1]$, so that $L'$ commutes with suspension. The monoidal structure of $L'$ is defined for $C\in\Ch^-(\P_{d,\k})$ and $D\in\Ch^-(\P_{e,\k})$ as the composite:
\begin{align*}\H(C^\sharp, \C S^d_{K(1)})\otimes \H(D^\sharp, \C S^e_{K(1)})& \xrightarrow[]{\otens} \H(C^\sharp\otimes D^\sharp, \C S^d_{K(1)}\otimes \C S^e_{K(1)})\\& \to\H(C^\sharp\otimes D^\sharp, \C (S^d\otimes S^e)_{K(1)}) \\
&\to\H(C^\sharp\otimes D^\sharp, \C (S^{d+e})_{K(1)})\;, 
\end{align*}
where the second map is induced by the shuffle map $\nabla: \C S^d_{K(1)}\otimes \C S^e_{K(1)}\to \C (S^d\otimes S^e)_{K(1)}$, and the last one is induced by the multiplication $S^d\otimes S^e\to S^{d+e}$. The unit morphism is the identity map in degree zero:
$$\k=\H(\k,\k)\to \H(\k,\k_{K(1)})=\k_{K(1)}\;.$$

To construct the isomorphism $\Sigma\circ \wtheta\simeq L(-;1)$ we compose the isomorphisms $L'\simeq L(-;1)$ and $\Sigma\circ \wtheta\simeq L'$ given by the two following two lemmas. This will finish the proof of proposition \ref{prop-main}.

\begin{lemma}
There is an isomorphism $L'\simeq L(-;1)$. This isomorphism commutes with suspension and with the monoidal structures. 
\end{lemma}
\begin{proof}
Recall from section \ref{subsubsec-internal} the isomorphism:
$$\H(F^\sharp,S^d_U)\simeq \H(\Gamma^{d,U},F)\simeq F_U\;.\qquad(\star)$$
Since this isomorphism is natural with respect to $F$ and $U$, we can replace $F$ by a complex $C\in \Ch^b(\P_{d,\k})$ and $U$ by the simplicial $\k$-module $K(1)$ to get an isomorphism of mixed complexes $\H(C^\sharp,S^d_{K(1)})\simeq C_{K(1)}$,
hence an isomorphism of complexes
$$L'(C)=\H(C^\sharp, \C S^d_{K(1)})=\MC\H(C^\sharp,S^d_{K(1)})\simeq \MC C_{K(1)}=L(C;1)\;.$$
It is obvious from the definition that this isomorphism commutes with suspension. If $d=0$ and $C=\k$, this isomorphism is the identity in degree zero, so it preserves the units of the monoidal functors $L'$ and $L(-;1)$. So it remains to check that the following diagram is commutative.
$$\xymatrix{
L'(C)\otimes L'(D)\ar[d]\ar[r]^-{\simeq}& L(C;1)\otimes L(D;1)\ar[d]^-{\oshuffle}\\
L'(C\otimes D)\ar[r]^-{\simeq} & L(C\otimes D;1)\;.
}\qquad \text{(D1)}$$
Since the functors are additive, we may restrict to the case $C\in \Ch^b(\P_{d,\k})$ and $D\in \Ch^b(\P_{e,\k})$. We proceed in two steps.

{\bf Step 1. A commutative diagram.} If $f\in \hom_{\V_\k}(U,W)$ and $g\in \hom_{\V_\k}(V,W)$ we denote by $\mathrm{mult}(f,g)$ the following composite, where the first map is induced by $f$ and $g$ and the second map is induced by the multiplication $S^d\otimes S^e\to S^{d+e}$:
$$\mathrm{mult}(f,g): S^d_{U}\otimes S^e_V\to S^d_{W}\otimes S^e_W \to S^{d+e}_W\;. $$
Then one readily checks that for all $F\in\P_{d,\k}$ and $G\in \P_{e,\k}$, the isomorphisms $(\star)$ fit into a commutative diagram:
$$\xymatrix{
\H(F^\sharp,S^d_U)\otimes \H(G^\sharp,S^d_V)\ar[d]^-{\otimes}\ar[rr]^-{\simeq} && F_U\otimes G_V\ar[dd]^-{F_f\otimes F_g}\\
\H(F^\sharp\otimes G^\sharp,S^d_U\otimes S^d_V)\ar[d]^-{\H(F^\sharp,G^\sharp,\mathrm{mult}(f,g))} &&\\
\H(F^\sharp\otimes G^\sharp, S^{d+e}_W)\ar[rr]^-{\simeq}&& F_W\otimes G_W \;.
}\qquad\text{(D2)}$$

{\bf Step 2. Proof of the commutativity of diagram (D1).} Restriction of diagram (D1) to the indices $i,j,p,q$ yields the following diagram.
$$\xymatrix{
\H((C^i)^\sharp,S^d_{K(1)_p})\otimes \H((D^j)^\sharp,S^d_{K(1)_q})\ar[d]\ar[r]^-{\simeq}_-{(\star)\otimes(\star)}& C^i_{K(1)_{p}}\otimes D^j_{K(1)_{q}}\ar[d]\\
\H((C^i)^\sharp\otimes (D^j)^\sharp,S^{d+e}_{K(1)_{p+q}})\ar[r]^-{\simeq}_-{(\star)}& C^i_{K(1)_{p+q}}\otimes D^j_{K(1)_{p+q}}\;.}
\quad \text{(D3)}$$
So, to prove the commutativity of diagram (D1), it suffices to prove that (D3) commutes for all indices $i,j,p,q$. 
Let us describe explicitly the vertical arrows in diagram (D3). By \cite[VIII Thm 8.8]{ML}, if $X$ is a simplicial object, the shuffle map $\nabla:X_p\otimes X_q\to X_{p+q}\otimes X_{p+q}$ is the sum over all $(p,q)$-shuffles $\mu$
$$\nabla=\sum \epsilon(\mu) f_\mu\otimes g_\mu $$
where $f_\mu:X_p\to X_{p+q}$ and $g_{\mu}:X_q\to X_{p+q}$ are the composites (with the $\sigma_i$ denoting the degeneracy operators of $X$): 
$$f_\mu=\sigma_{\mu(p+q)}\circ \sigma_{\mu(p+q-1)}\circ\cdots \circ \sigma_{\mu(p+1)}\,,\quad g_\mu=\sigma_{\mu(p)}\circ \cdots \circ \sigma_{\mu(1)}\;.$$
As a consequence, the right vertical arrow of (D3) is equal, up to a $(-1)^{jp}$ sign, to the sum over all $(p,q)$-shuffles $\mu$ of the maps $\epsilon(\mu)F_{f_\mu}\otimes G_{g_\mu}$, while the left vertical arrow of (D3) is equal up to a $(-1)^{jp}$ sign, to the sum over all $(p,q)$-shuffles $\mu$ of the postcomposition by $\epsilon(\mu)\mathrm{mult}(f_\mu,g_\mu)$:
$$ x\otimes y\mapsto \epsilon(\mu)\mathrm{mult}(f_\mu,g_\mu)\circ (x\otimes y)\;.$$

Since diagram (D2) commutes, by taking $U=K(1)_p$, $V=K(1)_q$, $W=K(1)_{p+q}$, $F=C^i$, $G=D^j$, $f=f_\mu$ and $g=g_\mu$, we get that postcomposition by $\epsilon(\mu)\mathrm{mult}(f_\mu,g_\mu)$ identifies through isomorphism $(\star)$ with $\epsilon(\mu)F_{f_\mu}\otimes G_{g_\mu}$. By summing over all $(p,q)$-shuffles, we obtain that diagram (D3) commutes.
\end{proof}

\begin{lemma}
There is an isomorphism $\Sigma\circ \wtheta\simeq L'$. This isomorphism commutes with suspension and with the monoidal structures.  
\end{lemma}
\begin{proof} 
We are actually going to prove an isomorphism of functors $\wtheta\simeq \Sigma^{-1}\circ L'$, compatible with suspensions and monoidal structures. If $C\in\Ch^b(\P_{d,\k})$ compatibility of $\H$ is with suspension (cf. section \ref{subsubsec-convention}) yields an isomorphism 
$$\psi_d:\wtheta(C)=\H(C^\sharp,\C S^d_{K(1)}[d])\simeq \H(C^\sharp,\C S^d_{K(1)})[d]= L'(C)[d]=(\Sigma^{-1}\circ L')(C)\;. $$
Gathering all these isomorphisms, we obtain a morphism $\psi:\wtheta\simeq \Sigma^{-1}\circ L'$.
This isomorphism commutes with suspension. Moreover, $\psi_0$ is the identity map, so this isomorphism preserves the units of the monoidal structures. So it remains to check that the following diagram commutes
 $$\xymatrix{
\wtheta(C)\otimes \wtheta(D)\ar[d]^-{\wlax}\ar[r]^-{\psi\otimes\psi}& (\Sigma^{-1}\circ L')(C)\otimes (\Sigma^{-1}\circ L')(D)\ar[d]\\
\wtheta(C\otimes D)\ar[r]^-{\psi} & (\Sigma^{-1}\circ L')(C\otimes D)\;.
}\qquad \text{(D1)}$$
By additivity of the functors appearing in the diagram, we may restrict to the case where $C\in\Ch^b(\P_{d,\k})$ and $D\in \Ch^b(\P_{e,\k})$. Since the map $f_{d,e}$ equals the composite $(\mathrm{mult}\,\circ \nabla)[d+e]\circ (\xi^{-1})_{d,e}$, diagram (D1) can be rewritten as follows.
$$\xymatrix{
\H(C^\sharp,\C S^d_{K(1)}[d])\otimes \H(D^\sharp,\C S^e_{K(1)}[e])\ar[d]^-{\otens}\ar[r]^-{\psi_d\otimes\psi_e} & \H(C^\sharp,\C S^d_{K(1)})[d]\otimes \H(D^\sharp,\C S^e_{K(1)})[e]\ar[d]^{(\xi^{-1})_{d,e}}\\
\H(C^\sharp\otimes D^\sharp,\C S^d_{K(1)}[d]\otimes \C S^e_{K(1)}[e])\ar[d]^-{\H(C^\sharp\otimes D^\sharp,(\xi^{-1})_{d,e})}& \H(C^\sharp,\C S^d_{K(1)})\otimes \H(D^\sharp,\C S^e_{K(1)})[d+e]\ar[d]^{\otens[d+e]}\\
\H(C^\sharp\otimes D^\sharp,\C S^{d}_{K(1)}\otimes \C S^e_{K(1)}[d+e]) \ar[d]^-{\H(C^\sharp\otimes D^\sharp,(\mathrm{mult}\,\circ \nabla)[d+e])}\ar[r]^-{\psi_{d+e}}&
\H(C^\sharp\otimes D^\sharp,\C S^{d}_{K(1)}\otimes \C S^e_{K(1)})[d+e]\ar[d]^-{\H(C^\sharp\otimes D^\sharp,\mathrm{mult}\,\circ \nabla)[d+e]}\\
\H(C^\sharp\otimes D^\sharp,\C S^{d+e}_{K(1)}[d+e])\ar[r]^-{\psi_{d+e}}\ar[r]^-{\psi_{d+e}}&
\H(C^\sharp\otimes D^\sharp,\C S^{d+e}_{K(1)})[d+e]
}$$
The upper square commutes by the compatibility properties of tensor products suspensions and $\H$, cf section \ref{subsubsec-convention}, and the lower square commutes since $\psi_{d+e}$ is a natural transformation. Thus diagram (D1) commutes.
\end{proof}

\section{Applications}\label{sec-applic}
\subsection{D\'ecalages}

\subsubsection{Recollections of Schur functors and Weyl functors}\label{subsubsec-schur}
If $\lambda=(\lambda_1,\dots,\lambda_k)$ is a partition of weight $\sum\lambda_i=d$, we denote by $\lambda'$ the conjugate partition. The Schur functor $S_\lambda$ associated to the partition $\lambda$ is defined as the image of the composite:
$$d_\lambda:\Lambda^{\lambda'}\hookrightarrow \otimes^d\xrightarrow[]{\sigma_\lambda}\otimes^d\twoheadrightarrow S^\lambda\;,$$
where the first map is the canonical inclusion of $\Lambda^{\lambda'}=\bigotimes_j \Lambda^{\lambda'_j}$ into $\otimes^d$, the last map is the canonical projection onto $S^\lambda=\bigotimes_i S^{\lambda_i}$ and the middle map is the isomorphism induced by sending $v_1\otimes\dots v_d$ to $v_{\sigma_\lambda(1)}\otimes\dots\otimes v_{\sigma_\lambda(d)}$ where $\sigma_\lambda\in\Si_d$ is the permutation defined as follows. Let $t_\lambda$ be the Young tableau with
standard filling: $1,\dots,\lambda_1$ in the first row, $\lambda_1+1,\dots,\lambda_2$ in the second row, etc. Then $\sigma_\lambda$, 
in one-line notation, is the row-reading of the conjugate tableau $t_{\lambda'}$.
As particular cases of Schur functors, we recover symmetric and exterior powers:
$$S_{(1^d)}=\Lambda^d\,,\quad S_{(d)}=S^d\,. $$

These Schur functors were first defined (in arbitrary characteristic) in \cite{ABW}. They are denoted there by a letter `$L$', but we prefer to denote them by the letter `$S$' and to keep the letter `$L$' for simple objects, as it is done in \cite{Jantzen,Martin}. Also, conjugate partitions are used in \cite{ABW} to index Schur functors, but we prefer the other convention, which agrees with \cite{Green,Jantzen,Martin,MacDonald,FultonHarris}. Schur functors have various notations and names, depending on the context. For the reader interested in reading the sources we have quoted, the following table provides the translation.

\medskip

\begin{center}
\begin{tabular}{c|c|c|c|c|c|c|c}
Reference & This article & \cite{FultonHarris} & \cite{MacDonald} & \cite{ABW} & \cite{Green} & \cite{Martin} &  \cite{Jantzen} \\
\hline
Notation & $S_\lambda$ & $\mathbb{S}_\lambda$ & $F(\lambda)$ & $L_{\lambda'}$ & $D_{\lambda,\k}$ & $M(\lambda)$ & $H^0(\lambda)$
\end{tabular}
\end{center}

\medskip

\begin{remark}
The notations in this table all occur in slightly different contexts. In \cite[Lecture 6]{FultonHarris} and in \cite[I, App.A]{MacDonald}, $\mathbb{S}_\lambda$ and $F(\lambda)$ refer to Schur functors defined over complex numbers. In \cite{Green,Martin}, the notations refer to modules over the Schur algebra, the table means that $S_\lambda(\k^n)$ coincides with $D_{\lambda,\k}$ and $M(\lambda)$ as a module over the Schur algebra $S(n,d)$ ($M(\lambda)$ is called a Schur module in \cite{Martin}, and $D_{\lambda,\k}$ is called a dual Weyl module in \cite{Green}). Finally, in \cite{Jantzen}, $H^0(\lambda)$ refers to a $GL_n$-module, so the table means that $S_\lambda(\k^n)$ coincides with $H^0(\lambda)$ as a $GL_n$-module. It is obvious that the objects $S_\lambda$, $\mathbb{S}_\lambda$, $L_{\lambda'}$ coincide. To see that $S_\lambda(\k^n)$ coincides with $M(\lambda)$, use the embedding embedding of $M(\lambda)\subset S^\lambda(\k^n)$ \cite[Example (1) p.73]{Martin}, and \cite[Thm II.2.16]{ABW}. Finally, $M(\lambda)$ and $H^0(\lambda)$ coincide by a theorem of James, cf \cite[Thm 3.2.6]{Martin} (see also \cite[Thm 3.4.1]{Martin}). 
\end{remark}

The Weyl functor $W_\lambda$ is the dual of the Schur functor: $W_\lambda=S_{\lambda}^\sharp$. It may be defined as the image of the composite
$$\Gamma^\lambda\hookrightarrow \otimes^d\xrightarrow[]{\sigma_{\lambda'}} \otimes^d\twoheadrightarrow \Lambda^{\lambda'}\;.$$ 
Just as in the case of Schur functors, these functors (or the associated representations) have various notations (and names) depending on the context. For example, $W_\lambda$ is called a coSchur functor in \cite{ABW}, and denoted there by $K_\lambda$ (see \cite[Prop II.4.1]{ABW} for the description as the dual of $S_\lambda$). No notation is used for Weyl functors in \cite{FultonHarris, MacDonald} because over a field of characteristic zero there is an isomorphism $W_\lambda\simeq S_\lambda$ \cite[Exercise 6.14]{FultonHarris}. Here is the conversion table.

\medskip

\begin{center}
\begin{tabular}{c|c|c|c|c|c}
Reference & This article & \cite{ABW} & \cite{Green} & \cite{Martin} &  \cite{Jantzen} \\
\hline
Notation & $W_\lambda$ & $K_\lambda$ & $V_{\lambda,\k}$ & $V(\lambda)$ & $V(\lambda)$
\end{tabular}
\end{center}

\medskip

Finally, in the context of highest weight categories, Schur functors $S_\lambda$ and Weyl functors $W_\lambda$ are respectively called costandard modules and standard modules, and they are often respectively denoted by $\nabla(\lambda)$ and $\Delta(\lambda)$ (although they are respectively denoted by $A(\lambda)$ and $V(\lambda)$ in \cite{CPS}). 
\subsubsection{Formality and d\'ecalages}

A complex $C\in \DD^b(\P_\k)$ is formal if there is an isomorphism $C\simeq H_*(C)$ in $\DD^b(\P_\k)$. The following lemma gives a sufficient condition for the formality of $\Theta^n F$.

\begin{lemma}\label{lm-formal}
Let $F\in\P_{\k}$. Assume that $H^i(\Theta^n(F))=0$ for $i\ne 0$. Then there is an isomorphism in the derived category $$H^0(\Theta^n F)\simeq \Theta^n(F)\;.$$
\end{lemma}
\begin{proof}
Observe that by construction, $\Theta^n(F)$ is isomorphic to a complex $C$ with $C^i=0$ if $i<0$. Hence we may produce the requested isomorphism as the composite $H^0(F)\simeq H^0(C)\simeq C\simeq \Theta^n(F)$. 
\end{proof}

In particular, if $F$ is a Schur functor, then $H^i(\Theta F)=0$ for $i>0$ by lemma \ref{lm-acyclic}, hence $\Theta(F)$ is formal. Moreover it is easy to compute $H^0(\Theta S_\lambda)$. 
\begin{lemma} Let $\lambda$ be a partition of weight $d$, and let $\k$ be a PID. There is an isomorphism
$H^0(\Theta S_\lambda)=W_{\lambda'}$.
\end{lemma}
\begin{proof}Let us fix a partition $\lambda=(\lambda_1,\dots,\lambda_m)$.
It is proved in \cite[Thm II.2.16]{ABW} that the Schur functor $S_{\lambda'}$ associated to the conjugate of $\lambda$ is the cokernel of the map
$[]_{\lambda}:\bigoplus_{\mu}\Lambda^\mu\to\Lambda^{\lambda} $, where the sum is taken over all tuples of positive integers of the form $(\lambda_1,\dots,\lambda_i+k,\lambda_{i+1}-k,\dots, \lambda_n)$ for all $1\le i\le n-1$ and all $1\le k\le \lambda_{i+1}$ and the restriction of $[]_{\lambda}$ to the summand indexed by the tuple $(\lambda_1,\dots,\lambda_i+k,\lambda_{i+1}-k,\dots, \lambda_n)$ is built by tensoring identities with the composite 
$$ \Lambda^{\lambda_{i+1}-k}\otimes \Lambda^{\lambda_{i+1}-k}\xrightarrow[]{\mathrm{comult}\otimes 1}\Lambda^{\lambda_{i+1}}\otimes \Lambda^k\otimes \Lambda^{\lambda_{i+1}-k}
\xrightarrow[]{1\otimes \mathrm{mult}}\Lambda^{\lambda_{i+1}}\otimes \Lambda^{\lambda_{i+1}}\;.$$
Hence $W_{\lambda'}$ is the kernel of the map $[]_{\lambda}^\sharp:\Lambda^{\lambda}\to \bigoplus_{\mu}\Lambda^\mu$. Similarly, it is proved in \cite[Thm II.3.16]{ABW} that the Weyl functor $W_\lambda$ is the cokernel of a similar map $[]'_\lambda:\bigoplus_\mu \Gamma^\mu\to\Gamma^\lambda$, hence $S_\lambda$ is the kernel of the map $([]'_\lambda)^\sharp$. Now tensors of exterior powers are $\H(\Lambda^d,-)$-acyclic, so 
$$H^0(\Theta S_\lambda)\simeq\ker(\H(\Lambda^d,([]_{\lambda'})^\sharp)=\ker (([]_{\lambda'})^\sharp)=W_{\lambda'}.$$
\end{proof}

So we have an isomorphism $W_{\lambda'}\simeq \Theta S_\lambda$ in $\DD^b(\P_{d,\k})$. Thus, theorem \ref{thm-main} yields the following result.
\begin{proposition}\label{prop-decalage}
Let $\k$ be a PID and let $\lambda$ be a partition of weight $d$ and $\lambda'$ be the dual partition. There are d\'ecalage isomorphisms (for all integers $i$ and $n$): 
$$L_iW_{\lambda'}(V;n)\simeq L_{i+d}S_\lambda(V;n+1)\;.$$
\end{proposition}
This proposition generalizes the d\'ecalage isomorphisms of Quillen \cite{Quillen2} and Bousfield \cite{Bousfield}:
$$L_i\Lambda^d(V;n)\simeq L_{i+d}S^d(V;n+1)\;,\quad L_i\Gamma^d(V;n)\simeq L_{i+d}\Lambda^d(V;n+1) $$
(These isomorphisms correspond to the cases $\lambda=(d)$ and $\lambda=(1,\dots,1)$). It also generalizes a result of Bott \cite{Bott}, who computed derived functors of Schur functors in characteristic zero. Indeed, in characteristic zero $W_\lambda\simeq S_{\lambda}$ so proposition \ref{prop-decalage} yields the following result.
\begin{corollary}[\cite{Bott}]
If $\k$ is a field of characteristic zero, the derived functors of a Schur functor indexed by a partition $\lambda$ of weight $d$ are given by:
$$L_*(S_\lambda;n)\simeq S_{\lambda'}[-nd]\text{ if $n$ is odd, and }L_*(S_\lambda;n)\simeq S_{\lambda}[-nd]\text{ if $n$ is even.}$$
\end{corollary}

\subsection{Plethysms}

The study of plethysms, i.e. representations given by composites of functors, is a hard problem of representation theory (even over a field of characteristic zero, see \cite[Chap I]{MacDonald}). For example, the composition series of the functor $F\circ G$ is usually unknown, even when the functors $F$ and $G$ are well understood. In this section, we our simplicial model for iterated Ringel duality, namely:
$$\Theta^n C [-dn]\simeq \MC C_{K(n)}$$
to give explicit formulas for Ringel duals of plethysms.
\subsubsection{Composition by Frobenius twists.} Let $\k$ be a field of positive characteristic $p$. The $r$-th Frobenius twist functor $I^{(r)}$ is the subfunctor of $S^{p^r}$ generated by $p^r$-th powers:
$$I^{(r)}(V):=V^{(r)}=\langle v^{p^r}\;|\; v\in V\rangle \subset S^{p^r}(V)\;. $$
Plethysms of the form $F\circ I^{(r)}$ and $I^{(r)}\circ F$ are of central importance for the representation theory over finite fields, see e.g. \cite{CPSVdK, FS, FFSS, TVdK}. The following proposition was first proved in \cite{Chalupnik2}, relying on the computations of \cite{FFSS} and \cite{Chalupnik1}. With our simplicial model, we can give an elementary proof.
\begin{proposition}[{compare \cite[Prop 2.6]{Chalupnik2}}]\label{prop-chal}
Let $\k$ be a field of positive characteristic $p$. There are isomorphisms, natural with respect to $C\in\DD^b(\P_{d,\k})$:
\begin{align*}
\text{(i)}&& \Theta^n(C\circ I^{(r)}) \simeq  \left( \Theta^n(C)\circ I^{(r)} \right) [dn(p^r-1)]\;, \\
\text{(ii)}&& \Theta^n(I^{(r)}\circ C) \simeq  \left( I^{(r)}\circ\Theta^n(C) \right) [dn(p^r-1)]\;.
\end{align*}
\end{proposition}
\begin{proof}
We shall prove (i), the proof of the second statement is similar. Since $I^{(r)}(V\otimes W)\simeq I^{(r)}(V)\otimes I^{(r)}(W)$ we have isomorphisms of mixed complexes:
$$(C\circ I^{(r)})_{K(n)}\simeq (C_{K(n)^{(r)}})\circ I^{(r)}\;.$$
Now $I^{(r)}(\k)\simeq \k$ and $I^{(r)}$ is additive, hence exact as a functor from $\k$-vector spaces to $\k$-vector spaces, so the complex of $\k$-vector spaces $\C K(n)^{(r)}$ has homology concentrated in homological degree $n$ and equal to $\k$ in this degree. So we have a quasi-isomorphism of complexes of $\k$-vector spaces $\N K(n)^{(r)}\simeq \k[-n]$, hence a homotopy equivalence between these complexes of $\k$-vector spaces. Now the Dold-Kan correspondence ensures that we have a homotopy equivalence $K(n)^{(r)}\simeq K(n)$ of simplicial $\k$-vector spaces. Thus the complexes of strict polynomial functors $\MC (C_{K(n)^{(r)}})$ and $\MC (C_{K(n)})$ are homotopy equivalent. Hence:
$$\Theta^n(C\circ I^{(r)})[-np^rd]\simeq \left(\MC(C_{K(n)})\right)\circ I^{(r)})\simeq  \left( \Theta^n(C)\circ I^{(r)} \right)[-nd]\;.$$
This finishes the proof of proposition \ref{prop-chal}.
\end{proof}

\subsubsection{Plethysm under a vanishing condition}
Now we go back to the case of an arbitrary PID $\k$. The following theorem provides an efficient way to compute Ringel duals of many plethysms.

\begin{theorem}\label{thm-pleth}Let $\k$ be a PID.
Let $G\in \P_{d,\k}$ and let $n$ be a positive integer such that $G$ satisfies the following vanishing condition 
$$ H^i(\Theta^n G)=0 \quad\text{for all $i>0$}\;.$$ 
Then for all $F\in \P_{\k}$, there is an isomorphism (in the derived category), natural with respect to $F$ and $G$:
$$\Theta^{n}(F\circ G)\simeq \Theta^{nd}(F)\circ H^0(\Theta^n G) \;.$$
\end{theorem}

Before we prove theorem \ref{thm-pleth}, we examine the statement and give a few consequences.
A lot of functors satisfy the vanishing condition appearing in theorem \ref{thm-pleth}. Let us give a list of examples.
\begin{enumerate}
\item If $n=1$, the functors satisfying the vanishing condition are the $\H(\Lambda^d,-)$-acyclic functors. Hence by lemma \ref{lm-acyclic}, they include symmetric powers, exterior powers, and more generally Schur functors. 
Sums, tensor products, or direct summands of functors satisfying the vanishing condition also satisfy it, as well as filtered functors whose graded pieces satisfy the vanishing condition\footnote{By the equivalence of categories $\P_{d,\k}\simeq S(n,d)\text{-mod}$, a functor $F$ is equivalent to a $S(n,d)$-module which is free of finite rank as a $\k$-module. Hence a filtration of $F$ must be finite. To prove the vanishing condition for $F$ from the vanishing condition for $\mathrm{Gr}(F)$, use the long exact sequences induced by short exact sequences at the level of $\Ext$s.}. For example, plethysms of the form $S^k\circ S^2$ and $S^k\circ \Lambda^2$ have a filtration whose graded pieces are Schur functors by a result of Boffi \cite{Boffi}, hence they satisfy the vanishing condition.
\item If $n=2$, the functors satisfying the vanishing condition are the $\H(S^d,-)$-acyclic functors. Hence they include tensor products of symmetric powers and more generally injective functors.
\item Tensor powers satisfy the vanishing condition for all $n$ (indeed, $\Theta^n (\otimes^d)\simeq\otimes^d$ by example \ref{exemple}). If the ground ring $\k$ is a field of positive characteristic, functors of degree strictly less than the characteristic  also satisfy the vanishing condition since they are direct summands of tensor powers.
\end{enumerate}
If $n$ equals $1$ or $2$, the homology of $\Theta^n(F\circ G)$ may be interpreted as extension groups. So we deduce from theorem \ref{thm-pleth} the following $\Ext$-computations involving plethysms. As usual, we denote by $\E^i(E,F\circ G)$ the functor assigning to $V$ the extension groups $\Ext^i_{\P_\k}(E^V, F\circ G)$. In particular, $\E^i(E,F\circ G)(\k)$ equals $\Ext^i_{\P_\k}(E, F\circ G)$.
\begin{corollary}\label{cor-Exthigher} Let $\k$ be a PID, and let $F\in \P_{e,\k}$. The following isomorphisms hold.
\begin{align}
&\E^*(S^{de}, F\circ \otimes^d)\simeq H^*(\Theta^{2d} F)\circ \otimes^d  \\
&\E^*(S^{de}, F\circ S^d)\simeq H^*(\Theta^{2d} F)\circ \Gamma^d\\
&\E^*(\Lambda^{de}, F\circ \otimes^d)\simeq H^*(\Theta^d F)\circ \otimes^d \\
&\E^*(\Lambda^{de}, F\circ S_\lambda)\simeq H^*(\Theta^d F)\circ W_{\lambda'}\label{eqn}
\end{align}
In particular, if $\lambda$ is not the partition $(1,\dots,1)$, that is if $S_\lambda$ is not an exterior power, then $W_{\lambda'}(\k)=0$ so that 
$$\Ext^*_{\P_\k}(\Lambda^{de}, F\circ S_\lambda)=0\;.$$
And if $S_\lambda=\Lambda^e$ is an exterior power, then $W_{\lambda'}(\k)=\Gamma^e(\k)=\k$ so that
$$\Ext^*_{\P_\k}(\Lambda^{de}, F\circ \Lambda^d)=H^*(\Theta^d F)(\k)\;.$$
\end{corollary}

\begin{remark} Corollary \ref{cor-Exthigher} shows that iterated Ringel duals $\Theta^d(F)$ for arbitrary large $d$ also have an interpretation in terms of extension groups in the category of strict polynomial functors.
\end{remark}

We finish by an application of the formulas of corollary \ref{cor-Exthigher}. In \cite{Boffi}, Boffi proves that the plethysm $S^k\circ S^2$ and $S^k\circ \Lambda^2$ have a universal filtration, that is a filtration defined over the ground ring $\Z$, whose graded pieces are Schur functors. Then he notices that `there is little hope' that plethysms of the form $S^k\circ S^d$ or $S^k\circ \Lambda^d$, for $d>2$ have similar universal filtrations and proves that this indeed fails for $d=3$. With the formulas of corollary \ref{cor-Exthigher}, we are able to prove the general non-existence result.
\begin{corollary}
Let $\k=\Z$ be the ground ring.
For $d>2$ and for $k\ge 2$, the plethysms $S^k\circ S^d$ and $S^k\circ \Lambda^d$ have no filtration whose graded pieces are Schur functors.
\end{corollary}
\begin{proof}We prove the case of $S^k\circ S^d$, the other case is similar. If $S^k\circ S^d$ had a filtration whose graded pieces are Schur functors, then Schur functors are $\H(\Lambda^{kd},-)$ acyclic, so we would have for all positive $i$:
$$\E^i_{\P_\Z}(\Lambda^{kd}, S^k\circ S^d)=0\;.$$
Moreover, we know that 
$$\E^0_{\P_\Z}(\Lambda^{kd}, S^k\circ S^d)=\H(\Lambda^{kd}, S^k\circ S^d)$$
has $\Z$-free values. In particular, $\E^*_{\P_\Z}(\Lambda^{kd}, S^k\circ S^d)$ has values in free graded $\Z$-modules (concentrated in degree zero). 
By corollary \ref{cor-Exthigher}(\ref{eqn}) this would mean that $H^*(\Theta^{d} S^k)\circ \Lambda^d$ has values in $\Z$-free graded modules (concentrated in degree zero). In particular  $L_{*}S^k(\Z;d)=L_*S^k(\Lambda^d(\Z^d);d)$ would be a graded free $\Z$-module (concentrated in degree $dk$). But the values of $L_{*}S^k(\Z;d)$ are well-known \cite{TouzeEML}, they are a well-identified direct summand (the direct summand of weight $k$) of the homology of the Eilenberg-Mac Lane space $K(\Z,d)$. In particular, if $d>2$ they have non trivial $p$-torsion when $p$ is a prime dividing $k$. This contradiction proves the result.
\end{proof}

\subsubsection{Proof of theorem \ref{thm-pleth}}
Let $X,Y\in s(\P_\k)$ be simplicial strict polynomial functors. We say that a morphism $f:X\to Y$ is a \emph{weak equivalence} if $\N f:\N X\to \N Y$  is a quasi-isomorphism of chain complexes (or equivalently if $\C F:\C X\to \C Y$ is a quasi-isomorphism). The proof of theorem \ref{thm-pleth} uses the following property of weak equivalences.
\begin{lemma}\label{lm-we}
Let $f\in \hom_{s(\P_\k)}(X,Y)$ be a weak equivalence. Then for all $F\in \P_\k$, the morphism $F(f)\in \hom_{s(\P_\k)}(F\circ X,F\circ Y)$ is also a weak equivalence. 
\end{lemma}
\begin{proof}
As proved in lemma \ref{lm-caracqis}(i), $f$ is a weak equivalence if and only if for all $V\in\V_\k$, the morphism of complexes of $\k$-modules $\N f_V:\N X(V)\to \N Y(V)$ is a quasi-isomorphism. By definition, functors in $\P_\k$ have $\k$-projective values, hence $\N X(V)$ and $\N Y(V)$ are complexes of projective $\k$-modules. Hence $\N f_V$ is a homotopy equivalence in $\Ch_{\ge 0}(\k\text{-mod})$ (although $\N f$ is \emph{not} a homotopy equivalence in $\Ch_{\ge 0}(\P_\k)$ in general). Using the Dold-Kan correspondence
$s(\k\text{-mod})\leftrightarrows \Ch_{\ge 0}(\k\text{-mod})$, we conclude that $f_V:X_V\to Y_V$ is a homotopy equivalence of simplicial $\k$-modules (we use here the simple observation that the normalized chain functors $\N$, and their inverses $\N^{-1}=\K$, commute with evaluation functors). Hence $F(f_V)=F(f)_V$ is a homotopy equivalence of $\k$-modules. In particular, $\N F(f)_V$ is a homotopy equivalence in $\Ch_{\ge 0}(\k\text{-mod})$, so that $\N F(f)$ is a quasi-isomorphism.
\end{proof}

\begin{proof}[Proof of theorem \ref{thm-pleth}]
First, by lemma \ref{lm-formal} and theorem \ref{thm-main}, the vanishing hypothesis yields an isomorphism in $\DD^b(\P_{d,\k})$:
$$H^0(\Theta^n G) [-nd]\simeq \C G_{K(n)}\simeq \N G_{K(n)}\;.$$
Such an isomorphism may be represented by a zig-zag in $\Ch_{\ge 0}(\P_{d,\k})$:
$$ H^0(\Theta^n G) [-nd]\leftarrow P \rightarrow \N G_{K(n)} \;,$$
in which the two arrows are quasi-isomorphisms, and $P$ is a complex of projective objects in $\P_{d,\k}$. We can take the complex $P$ in the middle as a complex concentrated in positive homological degrees because if $Q$ is a bounded below complex of projectives, with $H_i(Q)=0$ with $i<0$, then we can replace it by its truncation $Q'$ (with $Q'_i=0$ if $i<0$ and $Q'_0=\ker (Q_0\to Q_{-1})$), which is also a complex of projectives.
Applying the Dold-Kan functor $\K=\N^{-1}$ we get a zigzag of weak equivalences:
$$\K(H^0(\Theta^n G) [-nd])\leftarrow \K(P) \rightarrow \K\N G_{K(n)}\simeq G_{K(n)}\;.$$
Moreover, the explicit formula for $\K$ \cite[section 8.4]{Weibel} shows that 
$$\K(H^0(\Theta^n G) [-nd])\simeq \K(\k [-nd])\otimes H^0(\Theta^n G)=K(nd)\otimes H^0(\Theta^n G)\;.$$
Hence, by lemma \ref{lm-we}, applying $F$ to our zigzag of weak equivalences yields a zigzag of weak equivalences:
$$F_{K(nd)}\circ H^0(\Theta^n G)\leftarrow F\circ \K(P)\rightarrow F\circ (G_{K(n)})= (F\circ G)_{K(n)}\;. $$
Applying the chain functor, we thus get an isomorphism in $\DD^b(\P_{de,\k})$:
$$(\Theta^{nd}F)\circ H^0(\Theta^n G) [-nd]\simeq \Theta^n(F\circ G) [-nd]\;.$$
This finishes the proof of theorem \ref{thm-pleth}.
\end{proof}

\subsection{Block theory and vanishing}

\subsubsection{Brief recollections of block theory} In this section we recall basics of block theory for $\P_{d,\k}$. We refer the reader to \cite[Chap 5]{Martin} for further details.
Let $\k$ be a field of positive characteristic $p$. The simple objects of $\P_{d,\k}$ are classified by partitions of weight $d$. To be more specific, the Schur functors $S_{\lambda}$ have a simple socle, denoted by $L_\lambda$. The $L_\lambda$, for $\lambda$ a partition of weight $d$, are exactly the simple objects of $\P_{d,\k}$ \cite[Thm 3.4.2]{Martin}.

The blocks of $\P_{d,\k}$ are the equivalent classes of the set of simple objects of $\P_{d,\k}$ under the following equivalence relation. To simple objects $L_\lambda,L_\mu$ are equivalent if there is a sequence of simple functors 
$$L_\lambda=L_1,L_2,\dots,L_r=L_\mu $$ 
such that for all $i$, $\Ext^1_{\P_{d,\k}}(L_i,L_{i+1})\ne 0$. We let $\B_{d,\k}$ be the blocks of $\P_{d,\k}$.

For all functors $F$, and for all blocks $b\in \B_{d,\k}$, we denote by $F_b$ the sum of all subfunctors of $F$ whose composition factors belong to $b$. It is not hard to prove that $F=\oplus_{b\in\B_{d,\k}} F_b$. Moreover, if $F,G\in\P_{d,\k}$ and $b\ne b'$ are two blocks, then $\hom_{\P_{d,\k}}(F_b,G_{b'})=0$. So if we denote by $(\P_{d,\k})_b$ the full subcategory of $\P_{d,\k}$ whose objects are the $F_b$, for $F\in\P_{d,\k}$, the abelian category $\P_{d,\k}$ splits as a direct sum:
$$\textstyle\P_{d,\k} =\bigoplus_{b\in\B_{d,\k}}(\P_{d,\k})_b\;.$$
In particular, if no composition factors of $F$ and $G$ belong to the same block, then $\Ext^*_{\P_{d,\k}}(F,G)=0$.

The blocks of $\P_{d,\k}$ were determined by Donkin in \cite{DonkinHDim}. Two simple functors $L_\lambda,L_\mu$ lie in the same block if and only if the partitions $\lambda$ and $\mu$ have the same $p$-core (this is known as `Nakayama rule'). See \cite[6.2]{JamesKerber} for the definition of a $p$-core and \cite[Thm 5.3.1]{Martin} for another combinatorial formulation of the statement.

\subsubsection{Base change}
Let $\k$ be a field. By \cite{SFB}, there is an exact base change functor for strict polynomial functors
$$\P_{d,\Z}\to \P_{d,\k}\;,\quad F\mapsto F_{\k}\;.$$
The base change functor formalizes the fact that the integral equations defining $F$ may be interpreted in $\k$ (through the ring morphism $\Z\to\k$) to define an element of $\P_{d,\k}$. For example, the base change functor sends symmetric (resp. exterior, resp. divided, resp. tensor) powers of $\P_{d,\Z}$ to the symmetric (resp. exterior, resp. divided, resp. tensor) powers, viewed as elements of $\P_{d,\k}$.
In general, the functor $F_{\k}$ is characterized by the following condition. For all free and finitely generated $\Z$-modules $V$, 
$F(V)\otimes \k$ is naturally isomorphic to $F_{\k}(V\otimes \k)$. In particular we have isomorphisms of complexes of $\k$-modules
$$LF(V;n)\otimes \k\simeq LF_\k(V\otimes\k;n)\;. $$

\subsubsection{Vanishing results}
Now we indicate how to use block theory to get vanishing results for derived functors. We work here over the base ring $\k=\Z$. Let $F\in\P_{d,\Z}$. By base change and universal coefficient theorem, we have injective morphisms for all $i$:
$$L_iF(\Z^m;n)\otimes \mathbb{F}_p\hookrightarrow L_iF_{\mathbb{F}_p}(\mathbb{F}_p^m;n)\;.$$
So vanishing of $L_iF(\Z^m;n)\otimes \mathbb{F}_p$ follows from vanishing of $L_iF_{\mathbb{F}_p}(\mathbb{F}_p^m;n)$. Now our main theorem gives isomorphisms:
$$L_iF_{\mathbb{F}_p}(\mathbb{F}_p^m;1)\simeq \Ext^{d-i}_{\P_{d,\mathbb{F}_p}}(\Lambda^{d,\mathbb{F}_p^m}, F_{\mathbb{F}_p})\;, \;  L_iF_{\mathbb{F}_p}(\mathbb{F}_p^m;2)\simeq \Ext^{d-i}_{\P_{d,\mathbb{F}_p}}(S^{d,\mathbb{F}_p^m}, F_{\mathbb{F}_p})\;.$$
Thus, if no composition factor of $F_{\mathbb{F}_p}$ and $\Lambda^{d,\mathbb{F}_p^m}$ (or $S^{d,\mathbb{F}_p^m}$) belongs to the same 
block, then the $\Ext$-groups on the right hand side, hence the corresponding derived functors, must vanish.

As an example of this approach, we prove vanishing results for some derived functors of the Schur functor $S_{(d-1,1)}$ (which is more concretely described as the kernel of the multiplication $S^{d-1}\otimes S^1\to S^d$). To put our result in context, we first recall some easily obtained information on torsion of the derived functors of 
$S_{(d-1,1)}$.
\begin{lemma}\label{lm-tors}
Let $d\ge 3$ and let $n$ be a positive integer. For all $i$ and all $m$, the $\Z$-module $L_iS_{(d-1,1)}(\Z^m;n)$ may have $r$-torsion elements only for $r$ dividing $d!(d-1)!$. 
\end{lemma}
\begin{proof}
We know the derived functors of tensor products: $L_i\otimes^d(\Z^m;n)$ equals zero if $i\ne nd$ and $(\Z^m)^{\otimes d}$ if $i=nd$ (by a direct computation, or use example \ref{exemple} and theorem \ref{thm-main}).
Multiplication by $d!$, as a mormphism from $S^d$ to $S^d$ factors as the composite of the comultiplication $S^d\to \otimes^d$ and of the multiplication $\otimes^d\to S^d$. So, multiplication by $d!$ anihilates the torsion part of $L_iS^d(\Z^m;n)$ for all $i$, $n$, $d$, $m$. Similarly, multiplication by $(d-1)!$ anihilates the torsion part of $L_i(S^{d-1}\otimes S^1)(\Z^m;n)$. To prove lemma \ref{lm-tors}, it suffices to prove that multiplication by $(d-1)!d!$ anihilates the torsion part of $L_iS_{(d-1,1)}(\Z^m;n)$. This results from the information on the torsion of $S^d$ and $S^{d-1}\otimes S^1$ and the long exact sequence of derived functors
$$\dots\to L_{i+1}S^d(\Z^m;n)\xrightarrow[]{\partial}  L_iS_{(d-1,1)}(\Z^m;n)\to L_i(S^{d-1}\otimes S^1)(\Z^m;n)\to \dots $$
arising from the exact sequence $0\to S_{(d-1,1)}\to S^{d-1}\otimes S^1\to S^d\to 0$.
\end{proof}

We also recall that we can compute $L_iS_{(d-1,1)}(\Z^m;1)$ by d\'ecalage. It equals zero for $i\ne d$, and $W_{(2,1^{d-2})}(\Z^m)$ for $i=d$. So we are only interested in computing $L_iS_{(d-1,1)}(\Z^m;n)$ for $n\ge 2$.
We now give our vanishing result.
\begin{proposition}\label{prop-vanish}
Let $d\ge 3$ and let $p$ be a prime. If $d\ne 0\text{ mod }p$, then 
$$L_*S_{(d-1,1)}(\Z;2)\otimes\mathbb{F}_p =0 .$$
If $d\ne 0\text{ mod }p$ and if $d\ne 2\text{ mod }p$, then 
$$L_*S_{(d-1,1)}(\Z;3)\otimes\mathbb{F}_p =0 .$$
\end{proposition}
\begin{proof}
First, using d\'ecalage isomorphisms, we know that $L_*S_{(d-1,1)}(\Z;i+1)$ is isomorphic to $L_{*-d}W_{(2,1^{d-2})}(\Z,i)$. Thus we have to prove the vanishing of $L_*W_{(2,1^{d-2})}(\Z,i)\otimes\mathbb{F}_p$, for $i=1$ or $2$. To do this, it suffices to prove the vanishing of $L_*W_{(2,1^{d-2}), \mathbb{F}_p}(\mathbb{F}_p,i)$.

So, in the remainder of the proof, we work over the ground field $\mathbb{F}_p$.
Since the socle of a Schur functor is simple, all its composition factors belong to the same block. Taking duality, we obtain that all the composition factors of a Weyl functor belong to the same block. Moreover, simple functors are self-dual \cite[Thm 3.4.9]{Martin}, so the head (= cosocle = largest semi-simple quotient) of $W_{(2,1^{d-2})}$ is $L_{(2,1^{d-2})}$. So, the composition factors of $W_{(2,1^{d-2})}$ all belong to the same block as the simple functor $L_{(2,1^{d-2})}$. Thus, to prove proposition \ref{prop-vanish}, it suffices to prove that the partitions $(2,1^{d-2})$ and $(1^d)$ do not correspond to the same block if $d\ne 0\text{ mod }p$, and that the partitions $(2,1^{d-2})$ and $(d)$ do not correspond to the same block if $d\ne 0\text{ mod }p$ and $d\ne 2\text{ mod }p$.
Now the result follows from an easy application of Nakayama rule. Indeed, if $d= qp+r$ with $0\le r<p$, the $p$-core of $(1^d)$ is $(1^r)$ and the $p$-core of $(d)$ is $(r)$, while the $p$-core of $(2,1^{d-2})$ equals $(2,1^{r-2})$ if $r\ge 2$, $(2,1^{p-1})$ if $r=1$, and $(0)$ if $r=0$. 
\end{proof}

\begin{remark}
If we want to generalize proposition \ref{prop-vanish} to get vanishing results for $L_{*}F(\Z^m,i)$ for $m>1$, we need to determine the blocks which appear through the composition factors of $\Lambda^{d,\mathbb{F}_p}$ and $S^{d,\mathbb{F}_p}$. This can be easily done using Pieri rules \cite[Thm (3) p.168]{AB1}.
\end{remark}

\begin{remark}
It is interesting to compare proposition \ref{prop-vanish} with the discussion at the end of section 8.4 of \cite{BM}, where some computations of the derived functors of $S_{(d-1,1)}$ are discussed (the Schur functor $S_{(d-1,1)}$ is denoted by $J^d$ in this reference). In particular, it seems that the vanishing of $L_*S_{(d-1,1)}(\Z;3)\otimes\mathbb{F}_p$ we have obtained is optimal: there is $3$ and $5$ torsion in $L_*S_{(4,1)}(\Z;3)$, and there is $11$ and $13$ torsion in $L_*S_{(12,1)}(\Z;3)$.
\end{remark}

\subsection{An example of computation}
We finish by an example of elementary computation of derived functors involving Ringel duality, namely we compute $L_* S^p(V;n)$ over a field of positive characteristic $p$, and we retrieve from it  a computation of \cite{BM}.

\begin{proposition}\label{prop-calcSp}
Let $\k$ be a field of odd characteristic $p$ and let $n$ be a positive integer. The derived functors of $S^p$ satisfy
\begin{itemize}
\item[(1)] $L_{np} S^p(V;n)$ equals $\Lambda^p$ if $n$ is odd, and $\Gamma^p$ if $n$ is even.
\end{itemize}
If $n=1$ or $n=2$ the derived functors are trivial in the other degrees. If $n\ge 3$, the only nontrivial summands of $L_*S^p(V;n)$ are:
\begin{itemize}
\item[(2)] if $n$ is odd, and $1\le i\le (n-1)/2$, the summands
  $$L_{np-i(2p-2)+p}S^p(V;n)=L_{np-i(2p-2)+p-1}S^p(V;n)=V^{(1)}\;,$$
\item[(2')] if $n$ is even and $1\le i\le n/2-1$, the summands
$$L_{np-i(2p-2)+1}S^p(V;n)=L_{np-i(2p-2)}S^p(V;n)=V^{(1)}\;.$$
\end{itemize}
\end{proposition}
\begin{proof} We already know that $\Theta S^p=\Lambda^p$ and $\Theta \Lambda^p =\Gamma^p$. So we have to compute the iterated Ringel duals of $\Gamma^p$. Recall the exact Koszul complex (see e.g. \cite[Section 4]{FS}):
$$\Lambda^p\to \Lambda^{p-1}\otimes S^1\to \Lambda^{p-2}\otimes S^2\to \cdots\to \Lambda^1\otimes S^{p-1}\to S^p\;.\quad (*)$$
Its differentials are defined by using first the comultiplication $\Lambda^i\to \Lambda^{i-1}\otimes \Lambda^1$ (tensored by the identity of $S^{p-i}$) and then the multiplication $\Lambda^1\otimes S^{p-i}\to S^{p-i+1}$ (tensored by the identity of $\Lambda^{i-1}$).
Taking duality, we obtain another exact complex, the dual Koszul complex
$$\Gamma^p\to \Gamma^{p-1}\otimes \Lambda^1 \to \cdots\to \Gamma^1\otimes \Lambda^{p-1}\to \Lambda^p.$$
Splicing these two complexes together, we get a coresolution of $\Gamma^p$:
$$\Gamma^{p-1}\otimes \Lambda^1 \to \cdots\to \Gamma^1\otimes \Lambda^{p-1}\xrightarrow[]{\partial} \Lambda^{p-1}\otimes S^1\to \cdots\to S^p\;.$$
This is actually an injective coresolution of $\Gamma^p$ ($S^p$ is injective, and all the other objects are direct summands of $\otimes^p$). Let us denote by $T$ the truncation of the Koszul complex $(*)$ obtained by removing $\Lambda^p$ and $S^p$. Then the injective coresolution of $\Gamma^p$ can be informally written as $ T^\sharp \xrightarrow[]{\partial} T \to S^p$.

The Ringel dual of $\Gamma^p$ is the complex $ \H(\Lambda^p,T^\sharp \xrightarrow[]{\partial} T \to S^p)$. We can compute it explicitly. Indeed, 
$\H(\Lambda^p,\Lambda^i\otimes S^{p-i})=\Gamma^i\otimes \Lambda^{p-i}$
and $\H(\Lambda^p,-)$ sends the comultiplication $\Lambda^i\to \Lambda^{i-1}\otimes \Lambda^1$ to $\Gamma^i\to \Gamma^{i-1}\otimes \Gamma^1$ and the multiplication $\Lambda^1\otimes S^{p-i}\to S^{p-i+1}$ to $\Gamma^1\otimes \Lambda^{p-i}\to \Lambda^{p-i+1}$. Thus we obtain:
$$\H(\Lambda^p,T)=T^\sharp.$$
For $i<p$, $S^i$ is canonically isomorphic to $\Gamma^i$. The comultiplication $S^i\to S^{i-1}\otimes S^1$ identifies through this isomorphism with $\Gamma^i\to \Gamma^{i-1}\otimes \Gamma^1$, and the multiplication identifies $S^{i-1}\otimes S^1\to S^i$ identifies with $\Gamma^{i-1}\otimes \Gamma^1\to \Gamma^i$. So we also have:
$$\H(\Lambda^p,T^\sharp)=T.$$
Finally, $\H(\Lambda^p,-)$ sends the map $\partial$ onto the composite
$$d:\Lambda^1\otimes S^{p-1}\to S^p\to \Gamma^p \to \Gamma^{p-1}\otimes \Lambda^1.$$
(to see this embed $\Gamma^1\otimes \Lambda^{p-1}$ and $\Lambda^{p-1}\otimes S^1$ in $\otimes^p$, lift the map $\partial$, apply $\H(\Lambda^p,-)$ to the lift and use left exactness of $\H(\Lambda^p,-)$).
So we conclude that the Ringel dual of $\Gamma^p$ is given by 
$$\Theta\Gamma^p\simeq \left(T\xrightarrow[]{d} T^{\sharp}\to \Lambda^p\right)\simeq \left(T\xrightarrow[]{d} T^{\sharp}\xrightarrow[]{\partial} T\to S^p\right)\;. $$
To compute $\Theta^2\Gamma^p)$ we apply $\H(\Lambda^p,-)$ to the complex of injectives from the right hand side. The functor $\H(\Lambda^p,-)$ sends $d$ to $\partial$, so that $$\Theta^2 \Gamma^p\simeq\left(T^\sharp\xrightarrow[]{\partial} T\xrightarrow[]{d} T^{\sharp}\to \Lambda^p\right)\simeq\left(T^\sharp\xrightarrow[]{\partial} T\xrightarrow[]{d} T^{\sharp}\to T\to S^p\right)\;. $$
By induction we prove that:
$$\Theta^n\Gamma^p\simeq \left\{\begin{array}{l}
T^\sharp\xrightarrow[]{\partial} (T\xrightarrow[]{d} T^{\sharp})^{n/2}\to \Lambda^p\quad \text{if $n$ is even,}\\
(T\xrightarrow[]{d} T^{\sharp})^{(n+1)/2}\to \Lambda^p \quad\text{if $n$ is odd.}
\end{array}\right.$$

Now we compute the homology of these complexes. If $n$ is odd, $H^0(\Theta^n\Gamma^p)= \Lambda^p$ since it is the kernel of the map $\Lambda^{p-1}\otimes S^1\to \Lambda^{p-2}\otimes S^2$ appearing in the Koszul complex $(*)$. If $n$ is even $H^0(\Theta^n\Gamma^p)= \Gamma^p$ since it is the kernel of the map $\Gamma^{p-1}\otimes \Lambda^1\to \Gamma^{p-2}\otimes \Lambda^2$ appearing in the dual Koszul complex. This proves assertion (1). To prove assertion (2) and (3) we observe that the homology of the complex of the form:
$$ \dots\xrightarrow[]{\partial}T\xrightarrow[]{d}T^{\sharp}\xrightarrow[]{\partial} T \xrightarrow[]{d}T^{\sharp}\xrightarrow[]{\partial}\dots $$
is zero everywhere (as splices of exact complexes) except for the homology groups located at the source and the target of the maps $d$.
Indeed the map $S^p\to \Gamma^p$ has kernel $I^{(1)}$ and cokernel $I^{(1)}$ so we obtain that the homology groups at the source or the target of these maps equal 
$I^{(1)}$.
\end{proof}

Now let us work over the integers. For all free $\Z$-modules $V$, the composite $S^d(V)\hookrightarrow V^{\otimes d}\twoheadrightarrow S^d(V)$ equals multiplication by the scalar $d!$. But $L\otimes^d(V;n)$ is quasi isomorphic to $V^{\otimes d}[-nd]$, so by functoriality $d!$ annihilates the torsion part of $LS^d(V;n)$. Taking $d=p$, we see that the $p$-primary part of $L_iS^p(V;n)$ consists only of $p$-torsion. Hence, using base change and the universal coefficient theorem, we can retrieve from proposition \ref{prop-calcSp} the $p$-primary part of $L_iS^p(V;n)$. We obtain the following result.
\begin{corollary}\label{cor-calc-LSp}
Let $p$ be an odd prime, and let $V$ be a finitely generated free $\Z$-module. For $n\ge 3$, the $p$-primary part of $L_{*} S^p(V;n)$ is zero, with the following exceptions.
\begin{itemize}
\item If $n$ is odd, and $1\le i\le (n-1)/2$, the $p$-primary part of $L_{np-i(2p-2)+p-1}S^p(V;n)$ equals $(V\otimes\mathbb{F}_p)^{(1)}$.
\item If $n$ is even and $1\le i\le n/2-1$, the $p$-primary part of 
$L_{np-i(2p-2)}S^p(V;n)$ equals $(V\otimes\mathbb{F}_p)^{(1)}$.
\end{itemize}
\end{corollary}

In \cite[Thm 8.2]{BM}, Breen and Mikhailov compute the derived functors of the third Lie functor $\L^3$ which is the kernel of the map $S^2\otimes S^1\to S^3$ (so $\L^3$ is nothing but the Schur functor $S_{(2,1)}$).  Their result give the values of $L\L^3(V;n)$ over an arbitrary abelian group $V$. As an illustration of our techniques, we use corollary \ref{cor-calc-LSp} to recover their result when $V$ is free and finitely generated.

\begin{corollary}[Compare {\cite[Thm 8.2]{BM}}] Let $V$ be a free finitely generated $\Z$-module.
The derived functors $L_i\L^3(V;n)$ are trivial in all degrees except the following degrees.
\begin{itemize}
\item If $n$ is odd, 
$L_{3n}\L^3(V;n)$ is the kernel of the multiplication $\Lambda^2\otimes \Lambda^1\to \Lambda^3$ and for $1\le i\le (n-1)/2$,  $L_{3n-4i+1}\L^3(V;n)$ equals $(V\otimes\mathbb{F}_3)^{(1)}$.
\item If $n$ is even, $L_{3n}\L^3(V;n)=\L^3(V)$ and for $1\le i\le n/2$, $L_{3n-4i-1}\L^3(V;n)$ equals $(V\otimes\mathbb{F}_3)^{(1)}$.
\end{itemize}
\end{corollary}
\begin{proof} $\L^3(V)$ is formed by the elements of $S^2(V)\otimes S^1(V)$ of the form $ab\otimes c - bc\otimes a$. Thus, the composite (where $\tau_2$ is the transposition $(2,3)$)
$$\L^3\hookrightarrow S^2\otimes S^1\to \otimes^3 \xrightarrow[]{1-\tau_{2}}\otimes^3\to S^2\otimes S^1$$ (where $\tau_2$ is the transposition $(2,3)$) surjects on $\L^3$, and actually equals multiplication by the scalar $3$. So derived functors of $\L^3$ have only $3$-torsion. 

The short exact sequence $\L^3\hookrightarrow S^2\otimes S^1\twoheadrightarrow S^3$ induces a long exact sequence 
$$\cdots\to L_i(S^2\otimes S^1)(V;n)\to L_i S^3(V;n)\to L_{i-1}\L^3(V;n)\to \cdots\;.$$ 
Since $S^1\otimes S^2$ has no $3$-torsion (the composite $S^2\otimes S^1\to \otimes^3\to S^2\otimes S^1$ is multiplication by $2$), this long exact sequence splits into exact sequences (we use theorem \ref{thm-main} to identify derived functors in degree $3n$):
$$0\to L_{3n}\L^3(V;n)\to H^0\Theta^n (S^2\otimes S^1) \to H^0\Theta^n S^3\to  L_{3n-1}\L^3(V;n)\to 0\;.  $$
$$0\to L_j S^3(V;n)\otimes\mathbb{F}_3\to L_{j-1}\L^3(V;n)\to 0\;,\; \text{ for $j<3n$.} $$
We identify the map $H^0\Theta^n (S^2\otimes S^1) \to H^0\Theta^n S^3$ by successive applications of the functor $\H(\Lambda^3,-)$ to the multiplication $S^2\otimes S^1\to S^3$. Applying it once, we obtain the multiplication $\Lambda^2\otimes\Lambda^1\to \Lambda^3$. Applying it once again we obtain the multiplication $\Gamma^2\otimes\Gamma^1\to \Gamma^3$, applying it once again we obtain again the multiplication $\Lambda^2\otimes\Lambda^1 \to\Lambda^3$. Thus, if $n$ is odd,
$$L_{3n}\L^3(V;n)= \ker\{\Lambda^2\otimes\Lambda^1\to \Lambda^3\}\;,\;\text{ and } L_{3n-1}\L^3(V;n)=0\;.$$
Now, if $n$ is even, the injective morphism of algebras $S^*\to \Gamma^*$ yields a commutative diagram:
$$\xymatrix{
\Gamma^2\otimes\Gamma^1\ar[rr]^-{\mathrm{mult}}&&\Gamma^3\\
S^2\otimes S^1\ar@{^{(}->}[u]^-{(1)}\ar[rr]^-{\mathrm{mult}}&&S^3\ar@{^{(}->}[u]
}.$$
An elementary computation shows the map $(1)$ induces an isomorphism from $\L^3$ onto the kernel of $\Gamma^2\otimes\Gamma^1\to \Gamma^3$. Thus, for $n$ even:
$$L_{3n}\L^3(V;n)= \L^3(V)\;,\;\text{ and } L_{3n-1}\L^3(V;n)= (V\otimes\mathbb{F}_3)^{(1)}\;.$$
The other short exact sequences give $L_i\L^3(V;n)$ in lower degrees. 
\end{proof}

\section{Strict polynomial functors with non free values}\label{sec-arb}

In the previous sections, we have dealt with the category $\P_\k$ of strict polynomial functors having values in finitely generated free modules over the PID $\k$. This framework is sufficient for the applications of section \ref{sec-applic}, and it slightly simplifies the proofs since we don't have to derive tensor products. However, the case of functors with non-free values may be interesting, so in this section we briefly explain how to deal with this apparently more general case.

Let $\k$ be a PID. Let us denote by $\widetilde{\P}_{d,\k}$ the category of $\k$-linear functors from $\Gamma^d\V_\k$ to the abelian category of finitely generated $\k$-modules, and by $\widetilde{\P}_\k$ the direct sum of the categories $\widetilde{\P}_{d,\k}$.
Thus the category $\P_\k$ is a full exact subcategory of $\widetilde{\P}_\k$. 
The category $\widetilde{\P}_\k$ has enough projectives. A projective generator is provided by the functors $\Gamma^{d,V}$, $d\ge 0, V\in\V_\k$. We observe that it is also the projective generator of $\P_\k$. The following lemma is a formal consequence of this observation (together with the existence of finite projective resolutions \cite{DonkinHDim, AB} in the bounded case).
\begin{lemma}\label{lm-equiv}
Let the symbol $*$ stand for $-$ or $b$.
The functor $\P_\k\to \widetilde{\P}_\k$ induces equivalences of monoidal triangulated categories:
$$\DD^*(\P_\k)\simeq \DD^*(\widetilde{\P}_{\k})\;.$$
\end{lemma}
\begin{proof}
We have a commutative diagram of monoidal triangulated functors, where the two vertical arrows are equivalences of categories, with triangulated monoidal inverses:
$$\xymatrix{
\KK^*(\mathrm{Proj}(\P_\k))\ar[d]^-{\simeq}\ar[r]^-{=}& \KK^*(\mathrm{Proj}(\widetilde{\P}_\k))\ar[d]^-{\simeq}\\
\DD^*(\P_\k)\ar[r]&\DD^*(\widetilde{\P}_\k)
}\;.$$
\end{proof}  

The definition of the functors $L(-,n)$, $\Sigma$ and $\Theta$ can be generalized without change when working in $\widetilde{\P}_\k$. Now lemma \ref{lm-equiv} implies that our main theorem \ref{thm-main} stays valid with $\P_\k$ replaced by $\widetilde{\P}_\k$.

\section{Appendix: representations of $\k$-linear categories}\label{app}

We fix a commutative ring $\k$, and we denote by $\V_\k$ the category of finitely generated projective $\k$-modules and $\k$-linear maps.
We consider a category $\C$ enriched over $\V_\k$ (That is, $\hom$s are finitely generated projective $\k$-modules, and composition is bilinear).

In this appendix, we describe the relations between the following two categories.
\begin{itemize}
\item[(i)] The category $\C\text{-mod}$ of $\k$-linear representations of $\C$, that is the category of $\k$-linear functors $\C\to\V_\k$.
\item[(ii)] If $P\in\C$, then $\End_\C(P)$ is an algebra (for the composition), and we denote by $\End_\C(P)\text{-mod}$ the category of $\End_\C(P)$-modules, which are finitely generated and projective as $\k$-modules. 
\end{itemize}

Direct sums, products, kernels and cokernels in $\C\text{-mod}$ are computed objectwise in the target category. So $\C\text{-mod}$ inherits the structure of $\V_\k$. So it is an exact category, that is an additive category equipped with a collection of admissible short exact sequences \cite{Keller,Buehler}. To be more specific, the admissible short exact sequences are the pairs of morphisms $F\hookrightarrow G\twoheadrightarrow H$ which become short exact sequences of $\k$-modules after evaluation on the objects of $\C$. (If $\k$ is a field, $\C\text{-mod}$ is even an abelian category).

Similarly, the category $\End_\C(P)\text{-mod}$ is an exact category, the admissible short exact sequences are the pairs of morphisms $F\hookrightarrow G\twoheadrightarrow H$ which become short exact sequences of $\k$-modules after forgetting the action of $\End_\C(P)$.

The two categories are related via evaluation on $P$. If $F\in\C\text{-mod}$, the functoriality of $F$ makes the $\k$-module $F(P)$ into a $\End_\C(P)$-module. Thus we have an evaluation functor:
$$\begin{array}{ccc}
\C\text{-mod}&\to &\End_\C(P)\text{-mod}\\
F&\mapsto & F(P)
\end{array}. $$
This functor is additive and exact but in general it does not behave well with projectives. For example, $\End_\C(P)$ is a projective generator of $\End_\C(P)\text{-mod}$, whereas the projective functor $ \C^P:X\mapsto\hom_\C(P;X)$ need not be a projective generator of $\C\text{-mod}$. The following theorem gives a condition on $\C$ so that the evaluation functor is an equivalence of categories (compare \cite[lemma 3.4]{Kuhn}, where a similar statement is given in the case when $\k$ is a field and $\C$ is a category constructed from a functor with products, which by \cite[example 3.5]{Kuhn} encompasses the case of Schur algebras).

\begin{proposition}\label{prop-app}
Let $\C$ be a category enriched over $\V_\k$. Assume that there exists an object $P\in\C$ such that for all $X,Y\in\C$, the composition induces a surjective map
$$\hom_\C(X,P)\otimes \hom_\C(P,Y)\twoheadrightarrow \hom_\C(X,Y)\;.$$
Then the following holds.
\begin{enumerate}
\item[(i)] For all $F\in\C\text{-mod}$ and all $Y\in\C$ the canonical map $F(P)\otimes \hom_\C(P,Y)\to F(Y)$ is surjective.
\item[(ii)] The functor $\C^P:X\mapsto \hom_\C(P;X)$ is a projective generator of $\C\text{-mod}$.
\item[(iii)] Evaluation on $P$ induces an equivalence of categories $\C\text{-mod}\simeq \End_\C(P)\text{-mod}$.
\end{enumerate}
\end{proposition}
\begin{proof}
Let us prove (i). The canonical map is the map $x\otimes f\mapsto F(f)(x)$. Since $\hom_\C(Y,P)\otimes \hom_\C(P,Y)\twoheadrightarrow \hom_\C(Y,Y)$ is surjective, there exists a finite family of maps $\alpha_i\in \hom_\C(Y,P)$ and $\beta_i\in \hom_\C(P,Y)$ such that $\sum_i \beta_i\circ \alpha_i=\Id_Y$. For all $y\in F(Y)$ the element $\sum_i F(\alpha_i)(x)\otimes \beta_i\in F(P)\otimes \hom_\C(P,Y)$ is sent onto $y$ by the canonical map.

Let us prove (ii). The Yoneda isomorphism $\hom_{\C\text{-mod}}(\C^P,F)\simeq F(P)$ ensures that $\C^P$ is projective. Moreover (i) yields an epimorphism $F(P)\otimes \C^P\twoheadrightarrow F$. This epimorphism is admissible. Indeed, for all $X\in\C$, the kernel $K(X)$ of the map $F(P)\otimes \C^P(X)\twoheadrightarrow F(X)$ has values in $\V_\k$ (because the sequence of $\k$-modules $K(X)\hookrightarrow F(P)\otimes \C^P(X)\twoheadrightarrow F(X)$ is exact and the two other objects of the sequence are finitely generated and projective $\k$-modules). So the canonical map has a kernel in $\C\text{-mod}$.

To prove (iii), we have to show that the evaluation map is fully faithful and essentially surjective. The evaluation map fits into a commutative triangle (where the horizontal arrow is the Yoneda isomorphism and the diagonal arrow is the evaluation on the unit of $\End_\C(P)$):
$$\xymatrix{
\hom_{\C\text{-mod}}(\C^P,F)\ar[d]\ar[r]^-{\simeq}& F(P)\\
\hom_{\End_\C(P)\text{-mod}}(\End_\C(P),F(P))\ar[ru]_-{\simeq}
}.$$
Thus it induces an isomorphism. Now, by additivity of $\hom$s, we can extend this result to the functors of the form $V\otimes\C^P$, with $V\in\V_\k$. Finally, by (ii), any $G\in\C\text{-mod}$ has a presentation by functors of the form $V\otimes\C^P$. So by left exactness of $\hom_{\C\text{-mod}}(-,F)$ and $\hom_{\End_\C(P)\text{-mod}}(-,F(P))$, we obtain that evaluation on $P$ is fully faithful. 

It remains to prove that evaluation on $P$ is essentially surjective. Let $M$ be an $\End_\C(P)$, we may find a presentation of $M$ of the form: $V_2\otimes \End_\C(P)\xrightarrow[]{\psi} V_1\otimes \End_\C(P)\twoheadrightarrow M$ with $V_1,V_2\in\V_\k$. Since evaluation on $P$ is fully faithful, there exists a unique morphism $\phi:V_2\otimes\C^P\to V_1\otimes\C^P$ which coincides with $\psi$ after evaluation on $P$. We define a functor $F_M:\C\to\k\text{-mod}$ by  $F_M(X)=\mathrm{coker}\phi_X$ (the cokernel is taken in the category of $\k$-modules). Then $F_M(P)\simeq M$ is finitely generated and projective. Moreover for all $Y\in\C$, $F_M(Y)$ is a direct summand of $F_M(P)\otimes\hom_C(P,Y)$ (indeed, the canonical map $F_M(P)\otimes\hom_C(P,Y)\to F_M(X)$ is a retract of the map $x\mapsto \sum_iF_M(\alpha_i)(x)\otimes\beta_i$). So $F_M$ is actually a functor with finitely generated projective values. To sum up, we have found $F_M\in\C\text{-mod}$ whose evaluation on $P$ is isomorphic to $M$. That is, evaluation on $P$ is essentially surjective.
\end{proof}

The category $\Gamma^d\V_\k$ satisfies the hypotheses of proposition \ref{prop-app} with $P=\k^n$, $n\ge d$, see e.g. \cite[lemma 2.3]{TouzeClassical}. Hence we get a direct proof of the equivalence of categories $\Gamma^d\V_\k\text{-mod}\simeq S(n,d)\text{-mod}$, without appealing to the fact that $\Gamma^{d}\V_\k\text{-mod}$ is isomorphic to the category $\P_{d,\k}$ as defined by Friedlander and Suslin \cite{FS}.

\section*{Acknowledgements}
We thank Larry Breen, Nick Kuhn and Wilberd van der Kallen for their comments on a first version of this article.

\end{document}